\documentclass[a4paper]{article}

\usepackage{graphicx}
\usepackage{amsmath, amssymb, amsfonts, amsthm}
\usepackage{enumerate}
\usepackage[all]{xy}
\usepackage{mathtools}
\usepackage{hyperref}

\DeclareMathOperator{\tr}{tr} 
\DeclareMathOperator{\e}{e}   
\DeclareMathOperator{\sign}{sign}

\DeclareMathOperator{\dive}{div}

\newtheorem{thm}{Theorem}
\newtheorem{prop}[thm]{Proposition}
\newtheorem{lem}[thm]{Lemma}

\newcommand{\R}{{\mathbb R}}

\newcommand{\cX}{{\sf X}}
\newcommand{\cZ}{{\sf 0}}

\newcommand{\cdts}{\quad \cdots \quad}
\newcommand{\arr}[2]{\xrightleftharpoons[#2]{#1}}

\def\a{a}
\def\b{b}

\newcommand{\al}{\alpha}
\newcommand{\be}{\beta}

\makeatletter
\def\blfootnote{\xdef\@thefnmark{}\@footnotetext}
\makeatother

\begin{document}

\title{%
Planar S-systems: \\
Global stability and the center problem
}

\author{
Bal\'azs~Boros, Josef~Hofbauer, Stefan~M\"uller$^*$, \\ Georg Regensburger
}



\maketitle

\begin{abstract}
\noindent
S-systems are simple examples of power-law dynamical systems (polynomial systems with real exponents).
For planar S-systems, we study global stability of the unique positive equilibrium and solve the center problem.
Further, we construct a planar S-system with two limit cycles. \\[1ex]
\end{abstract}

\blfootnote{
\scriptsize

\noindent
{Bal\'azs~Boros $\cdot$ Josef~Hofbauer $\cdot$ Stefan~M\"uller} \\
Department of Mathematics, University of Vienna, Oskar-Morgenstern-Platz 1, 1090 Wien, Austria \\[1ex]
{Georg Regensburger} \\
Institute for Algebra, Johannes Kepler University Linz, Altenbergerstrasse 69, 4040 Linz, Austria \\[1ex]
$^*$Corresponding author \\
\href{mailto:st.mueller@univie.ac.at}{st.mueller@univie.ac.at}

}


\section{Introduction}

An S-system is a dynamical system on the positive orthant for which the right hand side is given by {\em binomials} (differences of monomials) with real exponents.
S-systems were introduced by Savageau~\cite{Savageau1969,Savageau1969b} in the context of biochemical systems theory.
For a recent review and an extensive list of references, see~\cite{Voit2013}.
In biochemical systems theory, one also considers dynamical systems given by {\em polynomials} with real exponents (power-law systems).  

As already observed in~\cite{Savageau1969b},
the binomial structure of an S-system allows to reduce the computation of positive equilibria to linear algebra by taking the logarithm.
In particular, it is easy to characterize when such a dynamical system has a unique positive equilibrium.
At the same time, 
already a {\em planar} S-system with a unique positive equilibrium may give rise to rich dynamical behaviour,
as demonstrated in this paper.

Andronov-Hopf bifurcations of planar S-systems are discussed in~\cite{lewis:1991},
and the first focal value is used to construct a stable limit cycle in~\cite{yin:voit:2008}. 
In fact, the first mathematical model of glycolytic oscillations by Selkov~\cite{Selkov68} is a planar S-system.
In previous work, we studied planar S-systems arising from a chemical reaction network (the Lotka reactions) with power-law kinetics.
A global stability analysis is given in~\cite{boros:hofbauer:mueller:2017},
and the center problem is solved in~\cite{boros:hofbauer:mueller:regensburger:2017}.
(For an interpretation of S-systems as dynamical systems
arising from chemical reaction networks with generalized mass-action kinetics~\cite{mueller:regensburger:2014},
we refer the reader to Appendix A.)

In this paper, we provide a global stability analysis of arbitrary planar S-systems (Section~\ref{sec:stability}).
In particular, we characterize the real exponents for which the unique equilibrium of a planar S-system is globally stable for all positive coefficients.
Further, we determine the parameters 
for which the unique equilibrium is a center (Section~\ref{sec:center_problem}).
In particular, we characterize global centers (Subsection~\ref{subsec:global_center}),
and finally we construct a planar S-system with two limit cycles bifurcating from a center (Subsection~\ref{subsec:limit}).
It remains open whether there exist planar S-systems with more than two limit cycles,
and we discuss a ``fewnomial version'' of Hilbert's 16th problem asking for an upper bound on the number of limit cycles for planar power-law systems in terms of the number of monomials. 
For an illustration of our analysis, we provide figures in Appendix B.

In the following section, we introduce planar S-systems, bring them into exponential form, and discuss the resulting symmetries.

\section{Planar S-systems} \label{sec:ss}

A planar S-system is given by
\begin{align} \label{ode:Ssys}
\dot x_1 &= \al_1 \, x_1^{g_{11}} x_2^{g_{12}} - \be_1 \, x_1^{h_{11}} x_2^{h_{12}}  , \\
\dot x_2 &= \al_2 \, x_1^{g_{21}} x_2^{g_{22}} - \be_2 \, x_1^{h_{21}} x_2^{h_{22}}   \nonumber
\end{align}
with $\alpha_1, \alpha_2, \beta_1, \beta_2 \in \mathbb{R}_+$ and $g_{11},g_{12},g_{21},g_{22},h_{11},h_{12},h_{21},h_{22} \in \R$.
Since we allow real exponents, we study the dynamics on the positive quadrant $\R_+^2$.

We assume that the ODE~\eqref{ode:Ssys} admits a positive equilibrium $(x_1^*,x_2^*)$,
and use the equilibrium to scale the ODE.
We obtain
\begin{align} \label{ode:Ssys_scaled}
\dot x_1 &= \gamma_1 \left( x_1^{g_{11}} x_2^{g_{12}} - x_1^{h_{11}} x_2^{h_{12}}  \right) , \\
\dot x_2 &= \gamma_2 \left( x_1^{g_{21}} x_2^{g_{22}} - x_1^{h_{21}} x_2^{h_{22}}  \right) , \nonumber
\end{align}
where $\gamma_1 = \al_1 (x_1^*)^{g_{11}-1} (x_2^*)^{g_{12}}$ and $\gamma_2 = \al_2 (x_1^*)^{g_{21}} (x_2^*)^{g_{22}-1}$.
The ODE~\eqref{ode:Ssys_scaled} admits the equilibrium $(1,1)$.

By a nonlinear transformation, we obtain a planar system with the origin as the unique equilibrium. 
In this exponential form, nullclines are straight lines and symmetries in the exponents can be exploited. 

\subsection{Exponential form}  \label{subsec:exponential}

Given the ODE~\eqref{ode:Ssys_scaled}, we perform the nonlinear transformations 
\[
x_1=\e^{\gamma_1 u}\quad \text{and} \quad x_2=\e^{\gamma_2 v}
\]
and obtain
\begin{align} \label{ode:exp}
\dot u &= \e^{\a_1 u + \b_1 v} - \e^{\a_2 u + \b_2 v} , \\
\dot v &= \e^{\a_3 u + \b_3 v} - \e^{\a_4 u + \b_4 v} , \nonumber
\end{align}
defined on $\R^2$, where
\begin{alignat}{3}\label{params_ab_gh}
\a_1 &= \gamma_1 (g_{11}-1), \quad & \b_1&=\gamma_2 g_{12}, \\
\a_2 &= \gamma_1 (h_{11}-1), & \b_2&=\gamma_2 h_{12}, \nonumber \\
\a_3 &= \gamma_1 g_{21}, & \b_3&=\gamma_2(g_{22}-1), \nonumber \\
\a_4 &= \gamma_1 h_{21}, & \b_4&=\gamma_2(h_{22}-1). \nonumber
\end{alignat}
The ODE~\eqref{ode:exp} admits the equilibrium $(0,0)$,
and the Jacobian matrix at $(0,0)$ is given by
\begin{equation} \label{eq:J}
J =
\begin{pmatrix}
\a_1-\a_2 & \b_1-\b_2 \\
\a_3-\a_4 & \b_3-\b_4
\end{pmatrix} .
\end{equation}

We abbreviate the ODE~\eqref{ode:exp} by its 8 parameters, more specifically, by the scheme
\begin{equation} \label{params}
\begin{pmatrix}
\a_1 & \a_2 & \a_3 & \a_4 \\
\b_1 & \b_2 & \b_3 & \b_4 
\end{pmatrix} .
\end{equation}

For any $\a,\b \in \R$, the ODE abbreviated by the parameter scheme 
\begin{equation*}
\begin{pmatrix}
\a_1-\a & \a_2-\a & \a_3-\a & \a_4-\a \\
\b_1-\b & \b_2-\b & \b_3-\b & \b_4-\b 
\end{pmatrix}
\end{equation*}
is obtained by multiplying the vector field in the ODE~\eqref{ode:exp} with $e^{-\a u-\b v}$ and is hence orbitally equivalent to \eqref{ode:exp}.
Thus the number of parameters could be reduced from 8 to 6.

Applying a uniform scaling transformation $(u,v) \mapsto (cu, cv)$ with $c > 0$ 
and rescaling time accordingly
is equivalent to dividing all parameters by~$c$.
Hence, 
the parameter space could be reduced to a 5-dimensional compact manifold. 
 

\subsection{Symmetry operations} \label{subsec:sym}

In the proofs of our main results (Sections~\ref{sec:stability} and~\ref{sec:center_problem}),
we exploit symmetries in the parameters, in order to avoid tedious case distinctions.

In fact, the family of ODEs~\eqref{ode:exp} is invariant under the symmetry group of the square (the dihedral group $D_4$) which consists of the following eight elements (rotations and reflections in $\R^2$):
\[
\begin{array}{llll}
 \mathbf{r}_0=\begin{pmatrix*}[r]  1 &  0 \\  0 &  1 \end{pmatrix*},
&\mathbf{r}_1=\begin{pmatrix*}[r]  0 & -1 \\  1 &  0 \end{pmatrix*},
&\mathbf{r}_2=\begin{pmatrix*}[r] -1 &  0 \\  0 & -1 \end{pmatrix*},
&\mathbf{r}_3=\begin{pmatrix*}[r]  0 &  1 \\ -1 &  0 \end{pmatrix*},\\[4ex]
 \mathbf{s}_0=\begin{pmatrix*}[r]  1 &  0 \\  0 & -1 \end{pmatrix*},
&\mathbf{s}_1=\begin{pmatrix*}[r]  0 &  1 \\  1 &  0 \end{pmatrix*},
&\mathbf{s}_2=\begin{pmatrix*}[r] -1 &  0 \\  0 &  1 \end{pmatrix*},
&\mathbf{s}_3=\begin{pmatrix*}[r]  0 & -1 \\ -1 &  0 \end{pmatrix*}.
\end{array}
\]

The question arises how these symmetry operations transform the ODE~\eqref{ode:exp}.
We start with $\mathbf{r}_1$, the rotation by $90^{\circ}$.
For $(U,V) = \mathbf{r}_1(u,v)$, that is, $U = -v, V = u$, we obtain
\begin{alignat*}{3}
\dot U &= -\dot v &&= \e^{\a_4 u + \b_4 v} - \e^{\a_3 u + \b_3 v} &&= \e^{-\b_4 U +  \a_4 V } - \e^{ -\b_3 U +  \a_3 V} , \\
\dot V &= \dot u &&= \e^{\a_1 u + \b_1 v} - \e^{\a_2 u + \b_2 v} &&= \e^{-\b_1 U +  \a_1 V } - \e^{ -\b_2 U +  \a_2 V} . \nonumber
\end{alignat*}
So $\mathbf{r}_1$ transforms the ODE~\eqref{ode:exp}, abbreviated by the parameter scheme~\eqref{params}, into the ODE abbreviated by
\begin{equation} \label{params_r1}
\begin{pmatrix*}[r]
- \b_4 & - \b_3 & - \b_1 & - \b_2  \\
  \a_4 &   \a_3 &   \a_1 &   \a_2 
\end{pmatrix*} .
\end{equation}

The other operations transform the parameter scheme~\eqref{params} as follows:

\begin{align}
\label{params_r2}
\mathbf{r}_2 &\colon \qquad
\begin{pmatrix*}[r]
- \a_2 & - \a_1 & - \a_4 & - \a_3 \\
- \b_2 & - \b_1 & - \b_4 & - \b_3 
\end{pmatrix*}\\[2ex]
\label{params_r3}
\mathbf{r}_3 &\colon \qquad
\begin{pmatrix*}[r]
  \b_3 &   \b_4 &   \b_2 &   \b_1 \\
- \a_3 & - \a_4 & - \a_2 & - \a_1 
\end{pmatrix*}\\[2ex]
\label{params_s0}
\mathbf{s}_0 &\colon \qquad
\begin{pmatrix*}[r]
  \a_1 &   \a_2 &   \a_4 &   \a_3 \\
- \b_1 & - \b_2 & - \b_4 & - \b_3 
\end{pmatrix*} \\[2ex]
\label{params_s1}
\mathbf{s}_1 &\colon \qquad
\begin{pmatrix*}[r]
\b_3 & \b_4 & \b_1 & \b_2 \\
\a_3 & \a_4 & \a_1 & \a_2 
\end{pmatrix*} \\[2ex]
\label{params_s2}
\mathbf{s}_2 &\colon \qquad
\begin{pmatrix*}[r]
- \a_2 & - \a_1 & - \a_3 & - \a_4 \\
  \b_2 &   \b_1 &   \b_3 &   \b_4 
\end{pmatrix*}\\[2ex]
\label{params_s3}
\mathbf{s}_3 &\colon \qquad
\begin{pmatrix*}[r]
- \b_4 & - \b_3 & - \b_2 & - \b_1 \\
- \a_4 & - \a_3 & - \a_2 & - \a_1 
\end{pmatrix*}
\end{align}

Note that the symmetry operations $\mathbf{r}_0, \mathbf{r}_2, \mathbf{s}_0, \mathbf{s}_2$ keep the roles of $\a_i$ and $\b_i$ (as coefficients of $u$ and $v$, respectively),
whereas the other four operations interchange them.
Only the subgroup consisting of $\mathbf{r}_0, \mathbf{s}_1$ keeps the signs of both $\a_i$ and $\b_i$.

Finally, the time reversal $t \mapsto -t$ transforms \eqref{ode:exp} into
\begin{equation}\label{params_tr}
\begin{pmatrix}
\a_2 & \a_1 & \a_4 & \a_3 \\
\b_2 &  \b_1 & \b_4 &  \b_3 
\end{pmatrix} .
\end{equation}


\section{Global stability} \label{sec:stability}

Ultimately, we are interested in stability properties of the unique positive equilibrium of the ODE~\eqref{ode:Ssys}.
Let $G=(g_{ij}) \in \R^{2 \times 2}$ and $H=(h_{ij}) \in \R^{2 \times 2}$.
A short calculation shows that the condition
\begin{equation*} 
\det
(G-H)
\neq 0
\end{equation*}
ensures that the ODE~\eqref{ode:Ssys} admits a unique positive equilibrium for all given values of the positive parameters $\al_1$, $\al_2$, $\be_1$, $\be_2$.
Then $(1,1)$ is the unique positive equilibrium of the ODE~\eqref{ode:Ssys_scaled},
and $(0,0)$ is the unique equilibrium of the ODE~\eqref{ode:exp} with \eqref{params_ab_gh}.
On the other hand, if $\det(G-H)=0$, then the ODE~\eqref{ode:Ssys} admits either no equilibrium or infinitely many equilibria,
depending on the specific values of $\al_1$, $\al_2$, $\be_1$, $\be_2$.

We call an equilibrium (in $\R^2_+$) of the ODEs~\eqref{ode:Ssys} or \eqref{ode:Ssys_scaled}
or an equilibrium (in $\R^2$) of the ODE~\eqref{ode:exp} with \eqref{params_ab_gh} \emph{globally asymptotically stable}
if it is Lyapunov stable and from each initial condition the solution converges to the equilibrium.

Below, we characterize the parameters $G$ and $H$ (the real exponents) 
for which the resulting ODEs admit a unique equilibrium that is (globally) asymptotically stable for {\em all other} parameters (the positive coefficients).
To begin with, 
we state the obvious relation between the stability properties of the ODEs~\eqref{ode:Ssys}, \eqref{ode:Ssys_scaled}, and \eqref{ode:exp} with \eqref{params_ab_gh}.

\begin{prop}
Fix $G,H \in \R^{2 \times 2}$ with $\det(G-H)\neq0$. Then the following are equivalent:
\begin{enumerate}[{\rm(i)}]
\item The unique positive equilibrium $(x_1^*,x_2^*)$ of the ODE~\eqref{ode:Ssys} is (globally) asymptotically stable for all $\al_1, \al_2, \be_1, \be_2 > 0$.
\item The unique positive equilibrium $(1,1)$ of the ODE~\eqref{ode:Ssys_scaled} is (globally) asymptotically stable for all $\gamma_1, \gamma_2 > 0$.
\item The unique equilibrium $(0,0)$ of the ODE~\eqref{ode:exp} with \eqref{params_ab_gh} is (globally) asymptotically stable for all $\gamma_1, \gamma_2 > 0$.
\end{enumerate}
\end{prop}

In our main results,
we consider the ODE~\eqref{ode:exp} with \eqref{params_ab_gh}
and write the Jacobian matrix~\eqref{eq:J} with \eqref{params_ab_gh} as
\[
J = (G-H)
\begin{pmatrix}
\gamma_1 & 0 \\ 0 & \gamma_2
\end{pmatrix} .
\]
Clearly, $\sign J=\sign(G-H)$ and $\sign(\det J)=\sign(\det(G-H))$.
In particular, $\det J\neq0$ if and only if $\det(G-H)\neq0$.
First, we characterize asymptotic stability.

\begin{prop} \label{prop:loc_asy_stab}
Fix $G,H \in \R^{2 \times 2}$ with $\det(G-H)\neq0$
and let $J$ be the Jacobian matrix of the ODE~\eqref{ode:exp} with \eqref{params_ab_gh} at the origin.
Then the following are equivalent:
\begin{enumerate}[{\rm(i)}]
\item
The unique equilibrium $(0,0)$ of the ODE~\eqref{ode:exp} with \eqref{params_ab_gh} is asymptotically stable for all $\gamma_1, \gamma_2 > 0$.
\item
$\det J > 0$ and $\sign J$ equals one of the sign matrices
\begin{align*}
\begin{pmatrix} - & * \\ * & - \end{pmatrix},
\begin{pmatrix} 0 & + \\ - & - \end{pmatrix},
\begin{pmatrix} 0 & - \\ + & - \end{pmatrix},
\begin{pmatrix} - & - \\ + & 0 \end{pmatrix}, 
\begin{pmatrix} - & + \\ - & 0 \end{pmatrix}.
\end{align*}
In particular, these conditions are independent of $\gamma_1, \gamma_2 > 0$.
\end{enumerate}
\end{prop}
\begin{proof}
Statement (ii) implies $\det J > 0$ and $\tr J < 0$, and statement (i) follows by a theorem of Lyapunov.
By the same theorem (together with the assumption $\det J \neq 0$), 
statement (i) implies $\det J > 0$ and $\tr J \leq 0$ for all $\gamma_1$, $\gamma_2 > 0$.
The trace condition is equivalent to both diagonal entries of $J$ being non-positive.
However, both diagonal entries being zero makes the origin a center, as can be seen using the integrating factor $\e^{-a_1u-b_4v}$, cf.\ case S in Subsection~\ref{subsec:first}. 
The signs of the off-diagonal entries follow from $\det J>0$.
\end{proof}

In our main result, we characterize global stability.

\begin{thm} \label{thm:glob_asy_stab}
Fix $G,H \in \R^{2 \times 2}$ with $\det(G-H)\neq0$
and let $J$ be the Jacobian matrix of the ODE~\eqref{ode:exp} with \eqref{params_ab_gh} at the origin.
Then the following are equivalent:
\begin{enumerate}[{\rm(i)}]
\item The unique equilibrium $(0,0)$ of the ODE~\eqref{ode:exp} with \eqref{params_ab_gh} is globally asymptotically stable for all $\gamma_1, \gamma_2 > 0$.
\item $\det J > 0$ and either
\begin{enumerate}[\rm(a)]
\item $\sign J = \begin{pmatrix} - & * \\ * & - \end{pmatrix}$,
\item $\sign J = \begin{pmatrix} 0 & + \\ - & - \end{pmatrix}$ and $\a_3 \leq \a_1 = \a_2 \leq \a_4$,
\item $\sign J = \begin{pmatrix} 0 & - \\ + & - \end{pmatrix}$ and $\a_4 \leq \a_1 = \a_2 \leq \a_3$,
\item $\sign J = \begin{pmatrix} - & - \\ + & 0 \end{pmatrix}$ and $\b_1 \leq \b_3 = \b_4 \leq \b_2$, or
\item $\sign J = \begin{pmatrix} - & + \\ - & 0 \end{pmatrix}$ and $\b_2 \leq \b_3 = b_4 \leq b_1$.
\end{enumerate}
In particular, these conditions are independent of $\gamma_1, \gamma_2 > 0$.
\end{enumerate}
\end{thm}

The proof of Theorem~\ref{thm:glob_asy_stab} requires two auxiliary results, Lemmas~\ref{lem:divergence} and~\ref{lem:bounded_solutions}.
There we study the ODE~\eqref{ode:exp} without the substitutions \eqref{params_ab_gh}.
Recall that its Jacobian matrix is given by~\eqref{eq:J},
that is,
\begin{align*}
J =
\begin{pmatrix}
\a_1-\a_2 & \b_1-\b_2 \\ \a_3-\a_4 & \b_3-\b_4
\end{pmatrix} .
\end{align*}
Lemma~\ref{lem:divergence} on the non-existence of periodic solutions 
will also be useful in Subsection~\ref{subsec:first}, where we look for first integrals of the ODE~\eqref{ode:exp}.
Lemma~\ref{lem:bounded_solutions} on the boundedness of solutions 
will also be useful in Subsection~\ref{subsec:global_center}, where we solve the global center problem. 

\begin{lem} \label{lem:divergence}
Let $a_1 \leq a_2$, $b_3 \leq b_4$ with $(a_1-a_2,b_3-b_4) \neq (0,0)$.
Further, let $a_1 \leq a  \leq a_2$ and $b_3 \leq b \leq b_4$. Then,
\begin{enumerate}
\item[\rm{(a)}] the r.h.s.\ of the ODE~\eqref{ode:exp} multiplied by $\e^{-au-bv}$ has negative divergence, 
\item[\rm{(b)}] there is no periodic solution of the ODE~\eqref{ode:exp}.
\end{enumerate}
\end{lem}
\begin{proof}
Let $f$ denote the r.h.s.\ of the ODE~\eqref{ode:exp}.
Multiplying $f(u,v)$ by $h(u,v)=\e^{-\a u-\b v}$ yields a vector field with negative divergence, since
\begin{align*}
\frac{\dive(h f)}{h}(u,v)=& \; (\a_1 - \a) \e^{\a_1 u + \b_1 v} + (\a-\a_2) \e^{\a_2 u + \b_2 v}\\ 
&+ (\b_3 - \b) \e^{\a_3 u + \b_3 v} + (\b-\b_4) \e^{\a_4 u + \b_4 v} .
\end{align*}
By the Bendixson-Dulac test, (a) implies (b).
\end{proof}

\begin{lem} \label{lem:bounded_solutions}
Let $J$ be the Jacobian matrix of the ODE~\eqref{ode:exp} at the origin 
with $\det J > 0$.
The following statements provide conditions for the boundedness of all solutions of the ODE~\eqref{ode:exp} in positive time.
\begin{enumerate}
\item[\rm(a)] If $\sign J = \begin{pmatrix} - & * \\ * & -  \end{pmatrix}$, then boundedness holds.
\item[\rm(b1)] If $\sign J = \begin{pmatrix} + & + \\ - & -  \end{pmatrix}$, then boundedness implies $a_3 \leq a_2 < a_1 \leq a_4$.
\item[\rm(b2)] If $\sign J = \begin{pmatrix} 0 & + \\ - & -  \end{pmatrix}$, then boundedness is equivalent to $a_3 \leq a_2 = a_1 \leq a_4$.
\item[\rm(c1)] If $\sign J = \begin{pmatrix} + & - \\ + & -  \end{pmatrix}$, then boundedness implies $a_4 \leq a_2 < a_1 \leq a_3$.
\item[\rm(c2)] If $\sign J = \begin{pmatrix} 0 & - \\ + & -  \end{pmatrix}$,  then boundedness is equivalent to $a_4 \leq a_2 = a_1 \leq a_3$.
\item[\rm(d1)] If $\sign J = \begin{pmatrix} -  & - \\ + & + \end{pmatrix}$, then boundedness implies $b_1 \leq b_4 < b_3 \leq b_2$.
\item[\rm(d2)] If $\sign J = \begin{pmatrix} -  & - \\ + & 0 \end{pmatrix}$, then boundedness is equivalent to $b_1 \leq b_4 = b_3 \leq b_2$.
\item[\rm(e1)] If $\sign J = \begin{pmatrix} -  & + \\ - & + \end{pmatrix}$, then boundedness implies $b_2 \leq b_4 < b_3 \leq b_1$.
\item[\rm(e2)] If $\sign J = \begin{pmatrix} -  & + \\ - & 0 \end{pmatrix}$, then boundedness is equivalent to $b_2 \leq b_4 = b_3 \leq b_1$.
\end{enumerate}
\end{lem}

\begin{proof}
We start by proving (a). In order to prove the boundedness of the solutions in the case $\a_1 < \a_2$, $\b_3 < \b_4$, and $\det J > 0$, we consider all possible signs of $\a_3-\a_4$ and $\b_1-\b_2$ and the corresponding nullcline geometries. For phase portraits in the nine cases, see Figure~\ref{fig:streamplots_very_stable_region}. In two cases (top left and bottom right),
solutions may spiral around the origin. Since the divergence of a scaled version of the right-hand side of the ODE~\eqref{ode:exp} is negative (see Lemma~\ref{lem:divergence}), they can spiral inwards only (anti-clockwise and clockwise, respectively). In the other seven cases, two of the four regions bounded by the nullclines are forward invariant, hence solutions are ultimately monotonic in both coordinates and converge to the origin.

The symmetry operations (of the square) introduced in Subsection~\ref{subsec:sym}
preserve $\det J > 0$ and the boundedness of solutions.
Hence, statements (c), (d), and (e) follow from (b) by applying the operations $\mathbf{s}_0$ or $\mathbf{s}_2$, $\mathbf{s}_1$ or $\mathbf{s}_3$, and $\mathbf{r}_1$ or $\mathbf{r}_3$, respectively, and it suffices to prove~(b).

\[
\xymatrix{
\text{(b)}
\ar@(l,u)^>(0.6){\textbf{r}_0,\textbf{r}_2}  \ar[rr]^{\textbf{s}_0,\textbf{s}_2} \ar[dd]^{\textbf{r}_1,\textbf{r}_3} \ar[ddrr]^{\textbf{s}_1,\textbf{s}_3} && 
\text{(c)} \\
\\
\text{(e)} && \text{(d)}
}
\]

To prove (b1), first note that $a_3 \leq a_2$ follows from $a_1 \leq a_4$ by applying the operation $\mathbf{r}_2$. Since $a_2 < a_1$ follows from the definition of the sign matrix, it suffices to prove that $a_1 \leq a_4$ is necessary for the boundedness. Assume $a_1 > a_4$ and $a_1 > a_2$.
A short calculation shows that the set
\[
\{(u,v) \in \R^2 ~|~ u\geq u_0 \mbox{ and } \gamma u \leq v \leq \gamma u_0\}
\]
is forward invariant under the ODE~\eqref{ode:exp} if $\gamma < 0$, $|\gamma|$ is small enough, and $u_0 > 0$ is large enough.
All the solutions starting in this forward invariant set are monotonic in both coordinates and unbounded.
For an illustration, see the top panel in Figure~\ref{fig:boundedness_proof}. 

We now show the necessity of $a_3 \leq a_2 = a_1 \leq a_4$ for the boundedness in (b2). The same argument as in the proof of (b1) shows that it suffices to prove that $a_1 \leq a_4$ is necessary for the boundedness. Assume $a_1 > a_4$ and $a_1 = a_2$ and consider the auxiliary ODE
\begin{align} \label{ode:exp_without_monom3}
\dot u &= \e^{\a_1 u + \b_1 v} - \e^{\a_2 u + \b_2 v} , \\
\dot v &=  - \e^{\a_4 u + \b_4 v}, \nonumber
\end{align}
which can be solved by separation of variables.
For $v > 0$,
the curve given by
\begin{align} \label{eq:solution_curve_exp_without_monom3}
\frac{\e^{(b_1-b_4)v}-1}{b_1-b_4} - \frac{\e^{(b_2-b_4)v}-1}{b_2-b_4} = -\frac{\e^{(a_4-a_1)u}}{a_4-a_1}
\end{align}
is an orbit of the ODE~\eqref{ode:exp_without_monom3} with $u \to +\infty$, $v \to 0$ for $t \to \infty$.
All solutions of the ODE~\eqref{ode:exp} that start above this curve are monotonic in both coordinates and unbounded.
For an illustration, see the bottom panel in Figure~\ref{fig:boundedness_proof}.
If $b_1-b_4$ or $b_2-b_4$ is zero, replace $\frac{\e^{\al v}-1}{\al}$ by $v$ in  \eqref{eq:solution_curve_exp_without_monom3}.

It is left to show the sufficiency of $a_3 \leq a_2 = a_1 \leq a_4$ for the boundedness in (b2). One can use a Lyapunov function $V : \R^2 \to \R$ with $(\partial_1 V)(u,v) = -\e^{-a_1 u}(\e^{a_3 u}-\e^{a_4 u})$ and $(\partial_2 V)(u,v) = \e^{-b_4 v}(\e^{b_1 v}-\e^{b_2 v})$. Assuming $a_3 < a_4$ and $b_1 > b_2$ (recall the assumption on $\sign J$), the boundedness of the sublevel sets of $V$ is equivalent to $a_3 \leq a_1 \leq a_4$ and $b_2 \leq b_4 \leq b_1$, see Figure~\ref{fig:lyap_level_sets} for the illustration of the level sets of $V$. Thus, if in addition to $a_3 \leq a_1 \leq a_4$, the inequalities $b_2 \leq b_4 \leq b_1$ also hold, the boundedness of the solutions of the ODE~\eqref{ode:exp} follows.
In case the inequalities $b_2 \leq b_4 \leq b_1$ do not hold, we also need to take into account the sign structure of the vector field in order to conclude the boundedness of the solutions. If $b_2 \leq b_4 \nleq b_1$, the set
\begin{align*}
\{(u,v)\in\R^2~|~V(u,v)\leq c \text{ and } v \leq d\}
\end{align*}
is bounded and forward invariant for all $c$ and for all sufficiently large $d$. If $b_2 \nleq b_4 \leq b_1$, the set
\begin{align*}
\{(u,v)\in\R^2~|~V(u,v)\leq c \text{ and } v \geq d\}
\end{align*}
is bounded and forward invariant for all $c$ and for all sufficiently negative $d$. For an illustration of the constructed sets, see Figure~\ref{fig:lyap_made_bounded}.
\end{proof}

Finally, we prove our main result.

\begin{proof}[Proof of Theorem~\ref{thm:glob_asy_stab}]
We have to show that among the systems fulfilling condition (ii) in Proposition~\ref{prop:loc_asy_stab}
exactly those do not admit periodic or unbounded solutions that meet condition (ii) in the present theorem.

In fact, all systems fulfilling condition (ii) in Proposition~\ref{prop:loc_asy_stab} are covered by Lemma~\ref{lem:divergence}
and hence do not admit a periodic solution.
Now, statements (a), (b2), (c2), (d2), (e2) in Lemma~\ref{lem:bounded_solutions} characterize those systems that do not admit an unbounded solution.
\end{proof}


\section{The center problem} \label{sec:center_problem}

An equilibrium is a {\em center} if all nearby orbits are closed.

Our aim is to characterize all parameters $a_1$, $a_2$, $a_3$, $a_4$, $b_1$, $b_2$, $b_3$, $b_4$ for which the origin is a center of the ODE~\eqref{ode:exp}.
First, we look for first integrals, then we find centers of reversible systems, and indeed we prove that we have identified all possible centers.
Thereby we use that an equilibrium (of an analytic ODE) is a center if and only if all focal values (Lyapunov coefficients) vanish, 
see~\cite[Chapters~3.5 and~8.3]{kuznetsov:2004} or~\cite[Chapter~3.1]{romanovski:shafer:2009}.

Additionally, 
we characterize all the parameters for which the origin is a {\em global} center.
Finally,
we construct a system with two limit cycles.

Let $J$ be the Jacobian matrix of the ODE~\eqref{ode:exp} at the origin.
For the origin to be a center, it is a prerequisite that $\tr J = 0$ and $\det J > 0$.
(If $\det J = 0$, then the origin lies on a curve of equilibria.)
Hence, we assume these conditions throughout Subsections~\ref{subsec:first} and \ref{subsec:rev}.


\subsection{First integrals} \label{subsec:first}

We look for first integrals (constants of motion) for the ODE~\eqref{ode:exp}
and try an integrating factor of the form $e^{-\a u-\b v}$.
As we have seen in the proof of Lemma~\ref{lem:divergence}, the divergence is proportional to
\begin{align}\label{div}
(\a_1 - \a) \e^{\a_1 u + \b_1 v} - (\a_2 - \a) \e^{\a_2 u + \b_2 v}
+ (\b_3 - \b) \e^{\a_3 u + \b_3 v} - (\b_4 - \b) \e^{\a_4 u + \b_4 v}.
\end{align}

First, we consider
\begin{equation*}
\a_1 = \a_2 \text{ and } \b_3 = \b_4.
\quad \textbf{(case~S)}
\end{equation*}
Setting $\a = \a_1$ and $\b = \b_4$, all four terms vanish, and the system is integrable.
In fact, the ODE~\eqref{ode:exp} is orbitally equivalent to
\begin{align*}
\dot u &= \e^{(\b_1-\b_4) v} - \e^{(\b_2 -\b_4)v} , \\
\dot v &= \e^{(\a_3 -\a_1) u } - \e^{(\a_4 -\a_1) u }, \nonumber
\end{align*}
a system with separated variables.
This case has codimension 2 in the parameter space.

Next, we consider
\begin{equation*}
\a_1 = \a_3 \text{ and } \b_1 = \b_3.
\quad \textbf{(case~I1)}
\end{equation*}
The divergence \eqref{div} simplifies to
\begin{align*}
(\a_1 - \a + \b_3 - \b) \e^{\a_1 u + \b_1 v} 
- (\a_2 - \a) \e^{\a_2 u + \b_2 v}
- (\b_4 - \b) \e^{\a_4 u + \b_4 v} .
\end{align*}
Setting $\a = \a_2$ and $\b = \b_4$, the last two terms vanish, and the first term is zero due to $\tr J = 0$.
This case has codimension 3 in the parameter space.

The following cases can be treated in the same way (and have codimension 3 in the parameter space):
\begin{alignat*}{3}
\a_1 &= \a_4 \text{ and } \b_1 = \b_4 \quad
&& \textbf{(case~I2)} \\
\a_2 &= \a_4 \text{ and } \b_2 = \b_4
&& \textbf{(case~I3)} \\
\a_2 &= \a_3 \text{ and } \b_2 = \b_3
&& \textbf{(case~I4)}
\end{alignat*}
It is easy to see that case S is invariant under all symmetry operations.

On the other hand, cases I1--I4 can be obtained from each other by symmetry operations.
Below, we apply all symmetry operations to case I1:
\[
\xymatrix{
\text{I1} \ar@(l,u)^>(0.6){\textbf{r}_0,\textbf{s}_1} \ar[rr]^{\textbf{r}_3,\textbf{s}_0} \ar[dd]^{\textbf{r}_1,\textbf{s}_2} \ar[ddrr]^{\textbf{r}_2,\textbf{s}_3} && \text{I2} \\
\\
\text{I4} && \text{I3}
}
\]

The corresponding first integrals are displayed in Table \ref{table:first_integrals_S_I1_I2_I3_I4}.

\begin{table}
\begin{center}
\begin{tabular}{|c||c|}
\hline
case & first integral \\
\hline\hline
S & $\displaystyle \left(\frac{\e^{pu}}{p} - \frac{\e^{qu}}{q}\right) - \left(\frac{\e^{rv}}{r} - \frac{\e^{sv}}{s}\right)$, \quad
where $\begin{cases} p = a_3-a_1, q = a_4-a_1, \\ r = b_1-b_4, s = b_2-b_4 \end{cases}$ \\
\hline \hline
I1 & $\displaystyle +\frac{\e^{p(u-v)}}{p} + \frac{\e^{q u}}{q} - \frac{\e^{r v}}{r}$, \quad
where $\begin{cases} p = a_1 - a_2, q = a_4 - a_2, \\ r = b_2 - b_4 \end{cases}$ \\
\hline
I2 & $\displaystyle -\frac{\e^{p(u+v)}}{p} + \frac{\e^{q u}}{q} + \frac{\e^{r v}}{r}$, \quad
where $\begin{cases} p = a_1 - a_2, q = a_3 - a_2, \\ r = b_2 - b_3 \end{cases}$ \\
\hline
I3 & $\displaystyle +\frac{\e^{p(-u+v)}}{p} - \frac{\e^{q u}}{q} + \frac{\e^{r v}}{r}$, \quad
where $\begin{cases} p = a_1 - a_2, q = a_3 - a_1, \\ r = b_1 - b_3 \end{cases}$ \\
\hline
I4 & $\displaystyle -\frac{\e^{p(-u-v)}}{p} - \frac{\e^{q u}}{q} - \frac{\e^{r v}}{r}$, \quad
where $\begin{cases} p = a_1 - a_2, q = a_4 - a_1, \\ r = b_1 - b_4 \end{cases}$ \\
\hline
\end{tabular}
\vspace{4ex}
\caption{First integrals corresponding to cases S, I1, I2, I3, I4.
If $\al$ is zero in $\frac{\e^{\al z}}{\al}$ (in a first integral), replace $\frac{\e^{\al z}}{\al}$ by $z$.} \label{table:first_integrals_S_I1_I2_I3_I4}
\end{center}
\end{table}


\subsection{Reversible systems} \label{subsec:rev}

Let $R \colon \R^2 \to \R^2$ be a reflection along a line. A vector field $F \colon \R^2 \to \R^2$ (and the resulting dynamical system) is called reversible w.r.t.\ $R$ if $F\circ R = - R \circ F$. The following is a well-known fact, see e.g.\ \cite[Chapter~II, 4.6571]{nemytskii:stepanov:1960}, \cite[Chapter~3.5]{romanovski:shafer:2009}, 
or more generally \cite[Theorem~8.1]{devaney:1976}: an equilibrium of a reversible system which has purely imaginary eigenvalues and lies on the symmetry line of $R$ is a center.

The above definition can be generalized and the fact still holds:
A vector field (system) $F$ is reversible w.r.t.\ the reflection $R$ if $- R^{-1} \circ F \circ R = \lambda F$ with $\lambda \colon \R^2 \to \R_+$.
That is, if $F$ transformed by $R$ followed by time reversal is orbitally equivalent to $F$.

The ODE~\eqref{ode:exp} is reversible w.r.t.\ $\mathbf{s}_1$, the reflection along the line $u = v$, 
if the system transformed by $\mathbf{s}_1$ followed by time reversal is orbitally equivalent to the original system.
That is, if applying \eqref{params_tr} to \eqref{params_s1} is equivalent to the original scheme,
\begin{equation*}
\begin{pmatrix}
\b_4 & \b_3 & \b_2 & \b_1 \\
\a_4 & \a_3 & \a_2 & \a_1 
\end{pmatrix}
\sim
\begin{pmatrix}
\a_1 & \a_2 & \a_3 & \a_4 \\
\b_1 & \b_2 & \b_3 & \b_4 
\end{pmatrix} .
\end{equation*}

This holds if and only if there exist $\a,\b\in\R$ such that
\begin{alignat*} {3}
\a_1 - \a  &= \b_4, \quad & \b_1 - \b &= \a_4, \\
\a_2 - \a  &= \b_3,       & \b_2 - \b &= \a_3, \\
\a_3 - \a  &= \b_2,       & \b_3 - \b &= \a_2, \\
\a_4 - \a  &= \b_1,       & \b_4 - \b &= \a_1, 
\end{alignat*}
that is,
\begin{equation*}
\a_1 - \b_4 = \a_2 - \b_3 = \a_3 - \b_2 = \a_4 - \b_1
\end{equation*}
or, equivalently,
\begin{equation}\label{rev_s1}
\a_1+\b_1=\a_4+\b_4, \; \a_2+\b_2=\a_3+\b_3, \text{ and } \tr J=0 .
\quad \textbf{(case R1)}
\end{equation}

The ODE~\eqref{ode:exp} is reversible w.r.t.\ $\mathbf{s}_3$, the reflection along the line $u = -v$, 
if the system transformed by $\mathbf{s}_3$ followed by time reversal is orbitally equivalent to the original system.
That is, if applying \eqref{params_tr} to \eqref{params_s3} is equivalent to the original scheme,
\begin{equation*}
\begin{pmatrix}
- \b_3 & - \b_4 & - \b_1 & - \b_2 \\
- \a_3 & - \a_4 & - \a_1 & - \a_2 
\end{pmatrix}
\sim
\begin{pmatrix}
\a_1 & \a_2 & \a_3 & \a_4 \\
\b_1 & \b_2 & \b_3 & \b_4 
\end{pmatrix}.
\end{equation*}

This holds if and only if there exist $\a,\b\in\R$ such that
\begin{alignat*}{3} 
\a_1 - \a  &= -\b_3, \quad & \b_1 - \b &= -\a_3,  \\
\a_2 - \a  &= -\b_4,             & \b_2 - \b &= -\a_4, \\
\a_3 - \a  &= -\b_1,             & \b_3 - \b &= -\a_1, \\
\a_4 - \a  &= -\b_2,             & \b_4 - \b &= -\a_2, 
\end{alignat*}
that is,
\begin{equation*}
\a_1 + \b_3 = \a_2 + \b_4 = \a_3 + \b_1 = \a_4 + \b_2
\end{equation*}
or, equivalently,
\begin{equation}\label{rev_s3}
\a_1-\b_1=\a_3-\b_3, \; \a_2-\b_2=\a_4-\b_4, \text{ and } \tr J=0 .
\quad \textbf{(case R2)}
\end{equation}

The two families of reversible systems given by \eqref{rev_s1} and \eqref{rev_s3}, respectively, have codimension 3 in the parameter space.
The other two reflections, $\mathbf{s}_0$ and $\mathbf{s}_2$ (across the $u$- and $v$-axis), also lead to reversible systems, however, they are already covered by case S.

Cases R1 and R2 can be obtained from each other by symmetry operations.
Below, we apply all symmetry operations to case R1:
\[
\xymatrix{
\text{R1} \ar@(l,u)^{\textbf{r}_0,\textbf{r}_2,\textbf{s}_1,\textbf{s}_3} \ar[rr]^{\textbf{r}_1,\textbf{r}_3,\textbf{s}_0,\textbf{s}_2} &&  \text{R2}
}
\]

Finally, we remark that neither for R1 nor for R2 we were able to find a first integral. However, for systems that are in the intersection of R1 and R2, the functions
\begin{align*}
\left[1+\e^{r(u+v)}\right]\left[\e^{qu}+\e^{qv}\right]^{-\frac{r}{q}} \text{ and } 
\e^{-a_1 u - b_4 v}(\e^{qu}+\e^{qv})^{-\frac{q+r}{q}}
\end{align*}
serve as first integral and integrating factor, respectively, where $q = a_4-a_1$ and $r = a_3-a_1$.


\subsection{Main result} \label{subsec:main}

In Table \ref{table:cases}, we display the seven cases of centers we identified in Subsections \ref{subsec:first} and \ref{subsec:rev}.
Indeed these are all possible centers of the ODE~\eqref{ode:exp}.

\begin{table} 
\begin{center}
\begin{tabular}{|c||c|c|c|}
\hline
case & \multicolumn{2}{c|}{parameters} \\
\hline\hline
S & $\a_1=\a_2$ & $\b_3=\b_4$ \\ \hline \hline
I1 & $\a_1=\a_3$ & $\b_1=\b_3$ \\ \hline
I2 & $\a_1=\a_4$ & $\b_1=\b_4$ \\ \hline
I3 & $\a_2=\a_4$ & $\b_2=\b_4$ \\ \hline
I4 & $\a_2=\a_3$ & $\b_2=\b_3$ \\ \hline \hline
R1 & $\a_1+\b_1=\a_4+\b_4$ & $\a_2+\b_2=\a_3+\b_3$ \\ \hline
R2 & $\a_1-\b_1=\a_3-\b_3$ & $\a_2-\b_2=\a_4-\b_4$ \\ \hline
\end{tabular}
\vspace{4ex}
\caption{Special cases of the ODE~\eqref{ode:exp} having a center.
Additionally, in all cases $\tr J = \a_1 - \a_2 + \b_3 - \b_4  = 0$,  which is trivially fulfilled in case~S,
and $\det J = (\a_1-\a_2)(\b_3-\b_4)-(\a_3-\a_4)(\b_1-\b_2) > 0$.
} \label{table:cases}
\end{center}
\end{table}

\begin{thm} \label{thm:center_problem}
Let $J$ be the Jacobian matrix of the ODE~\eqref{ode:exp} at the origin,
that is,
\begin{align*}
J = \begin{pmatrix}
\a_1-\a_2 & \b_1-\b_2 \\ \a_3-\a_4 & \b_3-\b_4
\end{pmatrix}.
\end{align*}
The following statements are equivalent:

\begin{itemize}
\item[\rm 1.]
The origin is a center of the ODE~\eqref{ode:exp}.
\item[\rm 2.]
The eigenvalues of the Jacobian matrix at the origin are purely imaginary, that is, $\tr J=0$ and $\det J>0$,
and the first two focal values vanish.
\item[\rm 3.]
The parameter values $a_1$, $a_2$, $a_3$, $a_4$, $b_1$, $b_2$, $b_3$, $b_4$ belong to (at least) one of the seven cases S, I1, I2, I3, I4, R1, and R2 in Table \ref{table:cases}.
\end{itemize}
\end{thm}

\begin{proof}
1 $\Rightarrow$ 2: If $J$ has a zero eigenvalue, that is, $\det J=0$, then the origin lies on a curve of equilibria
and cannot be a center.
Hence, 
the eigenvalues of $J$ are purely imaginary,
and all focal values vanish.

2 $\Rightarrow$ 3:
For the computation of the first two focal values, $L_1$ and $L_2$,
and the case distinction implied by $\tr J=0$, $\det J>0$, and $L_1=L_2=0$,
see Subsection~\ref{subsec:case}.

3 $\Rightarrow$ 1:
For the cases S, I1, I2, I3, and I4 in Table~\ref{table:cases}, we have found first integrals in Subsection~\ref{subsec:first}.
The remaining cases R1 and R2 are reversible systems, see Subsection~\ref{subsec:rev}.
\end{proof}

\subsection{Computation of focal values and case distinction} \label{subsec:case}

Instead of the ODE~\eqref{ode:exp}, we consider
\begin{align} \label{ode:simple}
\dot u &= 1 - \e^{\a_2 u + \b_2 v} , \\
\dot v &= \e^{\a_3 u + \b_3 v} - \e^{\a_4 u + \b_4 v} . \nonumber
\end{align}
After performing the substitutions
\begin{alignat}{3} \label{eq:subst}
\a_2 &\to \a_2-\a_1 , & \quad & \b_2 \to \b_2-\b_1 , \\
\a_3 &\to \a_3-\a_1 , && \b_3 \to \b_3-\b_1 , \nonumber \\
\a_4 &\to \a_4-\a_1 , && \b_4 \to \b_4-\b_1 , \nonumber
\end{alignat}
the ODE~\eqref{ode:simple} is orbitally equivalent to \eqref{ode:exp}.
Using
\[
\tr J = - \a_2 + \b_3 - \b_4  = 0 ,
\]
we compute $\det J$ and the first two focal values, $L_1$ and $L_2$.
We find
\[
\det J = (\a_3-\a_4)\b_2-(\b_3-\b_4)^2
\]
and note that $\det J>0$ implies $\a_3 \neq \a_4$ and $\b_2 \neq 0$.
Further,
using the Maple program in \cite{Kuznetsova:2012},
we find
\begin{equation*}
L_1 = - \frac{\pi}{8} \, \frac{(\b_3 - \b_4) \left[ ( \a_3 \a_4 + \a_3 \b_4 - \a_4 \b_3 ) \b_2 - (\a_3 - \a_4) \b_3 \b_4 \right]}{\sqrt{\det J} \, \b_2} .
\end{equation*}
Expressions for $L_2$ (in case $L_1=0$) will be given below.

We show that all parameters $(\a_2,\b_2,\a_3,\b_3,\a_4,\b_4) \in \R^6$ in the ODE~\eqref{ode:simple} for which
\begin{gather*}
\tr J=L_1=L_2=0 \\
\text{and } \det J>0
\end{gather*}
belong to one of the seven cases in Table~\ref{table:simple}.

\begin{table} 
\begin{center}
\begin{tabular}{|c||c|c|c|}
\hline
case & \multicolumn{2}{c|}{parameters} \\
\hline\hline
S & $\a_2=0$ & $\b_3=\b_4$ \\ \hline \hline
I1 & $\a_3=0$ & $\b_3=0$ \\ \hline
I2 & $\a_4=0$ & $\b_4=0$ \\ \hline
I3 & $\a_2=\a_4$ & $\b_2=\b_4$ \\ \hline
I4 & $\a_2=\a_3$ & $\b_2=\b_3$ \\ \hline \hline
R1 & $\a_4+\b_4=0$ & $\a_2+\b_2=\a_3+\b_3$ \\ \hline
R2 & $\a_3-\b_3=0$ & $\a_2-\b_2=\a_4-\b_4$ \\ \hline
\end{tabular}
\vspace{4ex}
\caption{Special cases of the ODE~\eqref{ode:simple} having a center.
Additionally, in all cases $\tr J = - \a_2 + \b_3 - \b_4  = 0$,  which is trivially fulfilled in case~S,
and $\det J = (\a_3-\a_4)\b_2-(\b_3-\b_4)^2 > 0$.
} \label{table:simple}
\end{center}
\end{table}

To begin with, $L_1=0$ implies either
\begin{enumerate}
\item[(a)] $\b_3=\b_4$,
\item[(b)] $\b_2=\frac{(\a_3 - \a_4) \b_3 \b_4}{D}$, where $D=\a_3 \a_4 + \a_3 \b_4 - \a_4 \b_3\neq0$, or
\item[(c)] $D=0$ and either $\b_3 = 0$ or $\b_4=0$.
Equivalently, either $\b_3 = 0$ and $\a_3(\a_4+\b_4)=0$ or $\b_4=0$ and $\a_4(\a_3-\b_3)=0$.
That is, either
\begin{itemize}
\item[(c1)] $\b_3=0$, $\a_3=0$,
\item[(c2)] $\b_3=0$, $\a_4+\b_4=0$,
\item[(c3)] $\b_4=0$, $\a_4=0$, or
\item[(c4)] $\b_4=0$, $\a_3-\b_3=0$.
\end{itemize}
\end{enumerate}

Case~(a), where $\b_3=\b_4$ and $\a_2=0$ (due to $\tr J=0$), corresponds to case~S in Table~\ref{table:simple}.

In case~(b), where $D \neq 0$ (and $\b_3,\b_4\neq0$ due to $\b_2\neq0$), we find
\[
L_2 = - \frac{\pi}{288} \, \frac{(\b_3 - \b_4) (\a_4 + \b_4) (\a_3 - \b_3) (\a_3 - \b_3 + \b_4) (\a_4 + \b_4 - \b_3) (\a_3 \b_4 - \a_4 \b_3)^2}{\sqrt{\det J} \, D \, \b_3 \b_4} ,
\]
using the Maple program in \cite{Kuznetsova:2012}.
Now, $L_2=0$ implies that at least one of six factors is zero:
\begin{enumerate}[-]
\item
As shown above, the first subcase $\b_3-\b_4=0$ is covered by case~S in Table~\ref{table:simple}.
\item
The subcase $\a_4 + \b_4=0$ implies $D=-\a_4 \b_3=\b_3 \b_4$ and hence $\b_2=\a_3-\a_4$.
Adding $\a_2=\b_3-\b_4$ (due to $\tr J=0$) yields $\a_2+\b_2=\a_3+\b_3$,
and the situation is covered by case~R1. 
\item
The subcase $\a_3 - \b_3=0$ also implies $D=\a_3 \b_4=\b_3 \b_4$ and hence $\b_2=\a_3-\a_4$.
Using $\a_2=\b_3-\b_4$ (due to $\tr J=0$) yields $\a_2-\b_2=\a_4-\b_4$,
and the situation is covered by case~R2.
\item
The subcase $\a_3 - \b_3 + \b_4=0$ implies $D=(\a_3-\a_4)\b_4$ and hence $\b_2=\b_3$.
Using $\tr J=-\a_2+\b_3-\b_4=0$ yields $\a_2=\a_3$,
and the situation is covered by case~I4.
\item
The subcase $\a_4 + \b_4 - \b_3=0$ implies $D=(\a_3-\a_4)\b_3$ and hence $\b_2=\b_4$.
Using $\tr J=-\a_2+\b_3-\b_4=0$ yields $\a_2=\a_4$,
and the situation is covered by case~I3.
\item
Finally, the subcase $\a_3 \b_4 - \a_4 \b_3=0$ implies $D=\a_3 \a_4$, $\b_2=\frac{(\a_3 - \a_4) \b_3 \b_4}{\a_3 \a_4}$,
and hence
\[
\det J = \frac{(\a_3 - \a_4)^2 \b_3 \b_4 - \a_3 \a_4 (\b_3-\b_4)^2}{\a_3 \a_4} = \frac{(\a_3 \b_3-\a_4 \b_4)(\a_3 \b_4 - \a_4 \b_3)}{\a_3 \a_4} = 0
\]
which need not be considered further.
\end{enumerate}

Case~(c1), where $\a_3=0$ and $\b_3=0$, corresponds to case~I1 in Table~\ref{table:simple}.

In case~(c2), where $\a_4+\b_4=0$ and $\b_3=0$,
we find
\[
L_2 = - \frac{\pi}{288} \, \frac{\a_3 \a_4^2 (\a_3-\a_4) (\a_4+\b_2) (\a_3-\a_4-\b_2)}{\sqrt{\det J} \, \b_2} .
\]
Now, $L_2=0$ implies that at least one of five factors is zero:
\begin{enumerate}[-]
\item
The first subcase $\a_3=0$ (and $\b_3=0$) is covered by case~I1 in Table~\ref{table:simple}.
\item
The subcase $\a_4=0$ (and hence $\b_4=0$) is covered by case~I2.
\item
As mentioned above, the subcase $\a_3-\a_4=0$ implies $\det J\le0$
which need not be considered further.
\item
The subcase $\a_4+\b_2=0$ (and $\a_4+\b_4=0$) implies $\b_2=\b_4$.
Moreover, $\a_2=\b_3-\b_4=\a_4$ (due to $\tr J=0$, $\b_3=0$ and $\a_4+\b_4=0$),
and the situation is covered by case~I3.
\item 
It remains to consider the subcase $\a_3-\a_4-\b_2=0$.
Adding $\tr J=-\a_2+\b_3-\b_4=0$ and using $\a_4+\b_4=0$ yields $\a_2+\b_2=\a_3+\b_3$,
and the situation is covered by case~R1.
\end{enumerate}

Case~(c3), where $\a_4=0$ and $\b_4=0$, corresponds to case~I2 in Table~\ref{table:simple}.

Finally, in case~(c4), where $\a_3-\b_3=0$ and $\b_4=0$, we find
\begin{align} \label{eq:L2_case_c4}
L_2 = \frac{\pi}{288} \, \frac{\a_3^2 \a_4 (\a_3-\a_4) (\a_3-\b_2) (\a_3-\a_4-\b_2)}{\sqrt{\det J} \, \b_2} .
\end{align}
Again, $L_2=0$ implies that at least one of five factors is zero.
The resulting subcases are covered by cases~I1, I2, ($\det J\le0$), I4, and R2 in Table~\ref{table:simple}.

To obtain the case distinction for the ODE~\eqref{ode:exp}, we perform the substitutions~\eqref{eq:subst} in Table~\ref{table:simple}.
The result is displayed in Table~\ref{table:cases}. 


\subsection{Global centers} \label{subsec:global_center}

We say that the origin is a {\em global} center of the ODE~\eqref{ode:exp} if all orbits are closed and surround the origin.

\begin{thm} \label{thm:global_center}
Let the origin be a center of the ODE~\eqref{ode:exp}. Then it is a {\em global} center if and only if
\begin{align} \label{eq:min_min_max_max}
\min(a_3,a_4) \leq \min(a_1,a_2) &\leq \max(a_1,a_2) \leq \max(a_3,a_4) \text{ and} \\
\min(b_1,b_2) \leq \min(b_3,b_4) &\leq \max(b_3,b_4) \leq \max(b_1,b_2). \nonumber
\end{align}
\end{thm}
\begin{proof}
For the cases S, I1, I2, I3, I4, the theorem follows immediately by investigating the level sets of the first integrals, see Table \ref{table:first_integrals_S_I1_I2_I3_I4}.

Below, we will implicitly use the easily checkable fact that condition \eqref{eq:min_min_max_max} is invariant under any of the symmetry operations $\mathbf{r}_0$, $\mathbf{r}_1$, $\mathbf{r}_2$, $\mathbf{r}_3$, $\mathbf{s}_0$, $\mathbf{s}_1$, $\mathbf{s}_2$, $\mathbf{s}_3$.

It suffices to show the theorem for the case R1, because the case R2 then follows by applying any of the symmetry operations $\mathbf{r}_1$, $\mathbf{r}_3$, $\mathbf{s}_0$, $\mathbf{s}_2$. In the sequel, we consider only R1. Also, we can assume that the system under consideration is not in case S, and thus $\sign J$ is one of
\begin{align*}
\begin{pmatrix} + & + \\ - & - \end{pmatrix},
\begin{pmatrix} + & - \\ + & - \end{pmatrix},
\begin{pmatrix} - & - \\ + & + \end{pmatrix}, 
\begin{pmatrix} - & + \\ - & + \end{pmatrix}.
\end{align*}
The 1st and the 3rd of these four cases can be transformed to each other by $\mathbf{s}_1$ and $\mathbf{s}_3$. The same applies to the 2nd and the 4th. Thus, we restrict our attention to the cases 
\begin{align*}
\begin{pmatrix} + & + \\ - & - \end{pmatrix} \text{ and }
\begin{pmatrix} - & + \\ - & + \end{pmatrix}.
\end{align*}

Another short calculation shows that the two chains of inequalities in \eqref{eq:min_min_max_max} are equivalent for R1. Note also that in the case R1 the ODE~\eqref{ode:exp} can be written in the orbitally equivalent form
\begin{align} \label{ode:R1_only_a}
\dot u &= \e^{\a_1 u + \a_4 v} - \e^{\a_2 u + \a_3 v} , \\
\dot v &= \e^{\a_3 u + \a_2 v} - \e^{\a_4 u + \a_1 v} . \nonumber
\end{align}
Therefore, we have to show that
\begin{enumerate}[{\rm(i)}]
\item if $a_1 > a_2$ and $a_3 < a_4$ then the origin is a global center for the ODE~\eqref{ode:R1_only_a} if and only if $a_3 \leq a_2 \leq a_1 \leq a_4$ and
\item if $a_1 < a_2$ and $a_3 < a_4$ then the origin is a global center for the ODE~\eqref{ode:R1_only_a} if and only if $a_3 \leq a_1 \leq a_2 \leq a_4$.
\end{enumerate}
In both of the cases (i) and (ii), the necessity of the inequality chain between $a_1$, $a_2$, $a_3$, $a_4$ follows immediately from Lemma~\ref{lem:bounded_solutions}.

The sufficiency in the case (i) follows directly by taking into account the nullcline geometry, the sign structure of the vector field, the fact that all the orbits are symmetric w.r.t the $u=v$ line, and the easily checkable fact that $\dot u + \dot v < 0$ whenever $u > v$, while $\dot u + \dot v > 0$ whenever $u < v$, see the top panel in Figure~\ref{fig:glob_center_proof}. (The sign of $\dot u + \dot v$ equals the sign of $v-u$, because both of the differences $\e^{\a_1 u + \a_4 v}-\e^{\a_4 u + \a_1 v}$ and $\e^{\a_3 u + \a_2 v} - \e^{\a_2 u + \a_3 v}$ are nonpositive (respectively, nonnegative) for $u>v$ (respectively, for $u<v$) and $\dot{u}+\dot{v}$ can be zero only if $u=v$, because $a_1=a_4$ and $a_2=a_3$ would imply $\det J =0$.)

In case (ii), we consider $\dot{u}-\dot{v}$ instead of $\dot{u}+\dot{v}$. We cannot determine where exactly it is positive and negative. However, it is enough that we know that it is negative (respectively, positive) whenever both of $(a_1-a_3)u+(a_4-a_2)v$ and $(a_4-a_2)u+(a_1-a_3)v$ are negative (respectively, positive), see the bottom panel in Figure~\ref{fig:glob_center_proof}. Starting from an initial point with $\dot{u}<0$ and $\dot{v}>0$, the solution will cross the $u$-nullcline and enter the region, where $\dot{u}>0$ and $\dot{v}>0$. Then the solution will reach the region, where $\dot{u}-\dot{v}>0$. Afterwards, it hits the $v$-nullcline and then the $u$-nullcline, after which $\dot{u}<0$ and therefore the solution will reach the region, where $\dot{u}-\dot{v}<0$. From there, it will hit the $v$-nullcline again.
\end{proof}


Finally, we remark that the center is clockwise (respectively, anticlockwise) if and only if $a_3<a_4$ and $b_1>b_2$ (respectively, $a_3>a_4$ and $b_1<b_2$).


\subsection{Limit cycles} \label{subsec:limit}

For the ODE~\eqref{ode:exp},
we are also interested in asymptotic stability
when the trace of the Jacobian matrix vanishes,
that is, when linearization does not give any information.
In fact, using the (sign of the) first focal value computed in Subsection~\ref{subsec:case},
we characterize super- and subcritical Hopf bifurcations resulting in a stable or unstable limit cycle, see also \cite{lewis:1991}.

\begin{prop} \label{hopf}
For the ODE~\eqref{ode:exp}, let $\det J > 0$ and $\tr J = 0$ at the origin and
\begin{align*}
\ell_1 = - (\b_3 - \b_4) & \bigg[ (\a_3-\a_1) (\a_4-\a_1) + (\a_3-\a_1) (\b_4-\b_1) - (\a_4-\a_1) (\b_3-\b_1) \bigg. \\
& \bigg. - \frac{(\a_3 - \a_4) (\b_3-\b_1) (\b_4-\b_1)}{(\b_2-\b_1)} \bigg] .
\end{align*}
If $\ell_1 < 0$, the origin is asymptotically stable.
If $\ell_1 > 0$, it is repelling.

If we consider a one-parameter family of ODEs~\eqref{ode:exp}
along which the eigenvalues of the Jacobian matrix cross the imaginary axis with positive speed,
for example, with parameter $\mu = \tr J$,
then an Andronov-Hopf bifurcation occurs at $\mu=0$.
If $\ell_1 < 0$, the bifurcation is supercritical (and there exists an asymptotically stable closed orbit for small $\mu>0$).
If $\ell_1 > 0$, it is subcritical (and there exists a repelling closed orbit for small $\mu<0$).
\end{prop}

Further, we are interested in a degenerate Hopf or Bautin bifurcation resulting in {\em two} limit cycles,
see~\cite[Section~8.3]{kuznetsov:2004}.
Indeed, using the first two focal values computed in Subsection~\ref{subsec:case},
we construct an S-system with two limit cycles.

In particular, we consider case (c4) in~Subsection~\ref{subsec:case}:
we set $a_1=b_1=b_4 = 0$, $a_3 = b_3 = a_2$ and hence $\tr J=L_1=0$
and choose $a_2$, $b_2$, $a_4$ such that $L_2 <0$ (and $\det J > 0$)
with $L_2$ given by Equation~\eqref{eq:L2_case_c4},
for example, $a_2=-1$, $b_2=-2$, $a_4=4$.
By slightly decreasing $b_3$ and $a_2$ (thereby keeping $b_3=a_2$ and $\tr J = 0$),
we obtain $L_1 > 0$,
and the resulting system has a stable limit cycle.
Finally, by slightly increasing $a_2$ such that $\tr J < 0$,
we create a small unstable limit cycle via a subcritical Hopf bifurcation.

It remains open, whether the ODE~\eqref{ode:exp} admits more than two limit cycles.
In fact, one could formulate a ``fewnomial version'' of the second part of Hilbert's 16th problem for planar power-law systems defined on the positive quadrant:
Khovanskii~\cite{Khovanskii1991} gives an explicit upper bound on the number of nondegenerate positive solutions of $n$ generalized polynomial equations in $n$ variables
in terms of the number of distinct monomials; see also~\cite{Sotille2011}.
Similarly, we can ask for an upper bound on the number of limit cycles of planar power-law systems (with finitely many equilibria) in terms of the number of monomials. 

In analogy to the cyclicity problem (the local version of Hilbert's 16th problem), we can also ask for an upper bound on the number of limit cycles that can bifurcate from a center,
when we fix the number of monomials and their signs and perturb the positive coefficients and real exponents.
Our example shows that in the simplest case with two binomials this upper bound is at least two.
For a computational algebra approach to this question for planar polynomial systems with real or complex coefficients and integer exponents of small degree,
see~\cite{romanovski:shafer:2009}.


\subsection*{Acknowledgments}

BB and SM were supported by the Austrian Science Fund (FWF), project P28406.
GR was supported by the FWF, project P27229.

\subsection*{Supplementary material}

We provide a Maple worksheet
containing (i) the program from \cite{Kuznetsova:2012}
for the computation of the first two focal values
and (ii) the case distinction described in Section~\ref{subsec:case}.

The material is available at \url{http://gregensburger.com/softw/s-systems/}.

\appendix

\subsection*{Appendix A: S-systems as generalized mass-action systems}

Every planar S-system can be specified as a generalized mass-action system
in terms of~\cite{mueller:regensburger:2014} (based on \cite{mueller:regensburger:2012}).
In particular, it arises from a directed graph containing two connected components with two vertices and two edges each,
\begin{align*}
1 \arr{}{} 2, \\
3 \arr{}{} 4.
\end{align*}
To each vertex, one assigns a stoichiometric complex (either the zero complex~$\cZ$ or one of the molecular species~$\cX_1$ and $\cX_2$),
in particular, one specifies the reversible reactions
\begin{align*}
\cZ \arr{}{} \cX_1, \\
\cZ \arr{}{} \cX_2,
\end{align*}
representing the production and consumption of $\cX_1$ and $\cX_2$.

To each vertex, one further assigns a {\em kinetic-order} complex (a formal sum of the molecular species),
thereby determining the exponents in the power-law reaction rates,
and to each edge, one assigns a positive rate constant.
One obtains
\begin{align} \label{gmas}
g_{11} \cX_1 + g_{12} \cX_2 \cdts \cZ \arr{\al_1}{\be_1} \cX_1 \cdts h_{11} \cX_1 + h_{12} \cX_2 , \\
g_{21} \cX_1 + g_{22} \cX_2 \cdts \cZ \arr{\al_2}{\be_2} \cX_2 \cdts h_{21} \cX_1 + h_{22} \cX_2 , \nonumber
\end{align}
implying the reaction rates $v_{\cZ \to \cX_1} = \al_1 \, x_1^{g_{11}} x_2^{g_{12}}$, $v_{\cX_1 \to \cZ} = \be_1 \, x_1^{h_{11}} x_2^{h_{12}}$, etc.

The resulting S-system is given by
\begin{align*}
\dot x_1 &= \al_1 \, x_1^{g_{11}} x_2^{g_{12}} - \be_1 \, x_1^{h_{11}} x_2^{h_{12}}  , \\
\dot x_2 &= \al_2 \, x_1^{g_{21}} x_2^{g_{22}} - \be_2 \, x_1^{h_{21}} x_2^{h_{22}}  \nonumber
\end{align*}
with $\alpha_1, \alpha_2, \beta_1, \beta_2 \in \mathbb{R}_+$ and $g_{11},g_{12},g_{21},g_{22},h_{11},h_{12},h_{21},h_{22} \in \R$.

For mass-action systems, the deficiency (a nonnegative integer) plays a crucial role in the analysis of the dynamical behaviour.
For example, if the deficiency is zero, then periodic solutions are not possible.
For the {\em generalized} mass-action system~\eqref{gmas},
the stoichiometric deficiency~\cite{mueller:regensburger:2014} is given by
\[
\delta = 4-2-2 = 0 ,
\]
since there are $4$ vertices and $2$ connected components in the graph and the stoichiometric subspace has dimension $2$.
In contrast to mass-action systems with deficiency zero, this system gives rise to rich dynamical behaviour.

Analogously, every $n$-dimensional S-system can be specified as a generalized mass-action system in terms of~\cite{mueller:regensburger:2014} with deficiency zero.
In fact, every generalized mass-action (GMA) system in terms of biochemical systems theory (BST) can be specified as a generalized mass-action system in terms of~\cite{mueller:regensburger:2014}.
More specifically, every power-law dynamical system arises from a generalized chemical reaction network,
that is, a digraph without self-loops and two functions  assigning to each vertex a stoichiometric complex and to each source vertex a kinetic-order complex.
Thereby, complexes need not be different, as in the case of the zero complex $\cZ$ in the generalized mass-action system~\eqref{gmas}.

%
%


\subsection*{Appendix B: Figures}

In the following figures,
we illustrate our analysis of the ODE \eqref{ode:exp}.
Thereby, the red line is the $u$-nullcline, $a_1 u + b_1 v = a_2 u + b_2 v$, while the green line is the $v$-nullcline, $a_3 u + b_3 v = a_4 u + b_4 v$.

Figures \ref{fig:streamplots_very_stable_region}, \ref{fig:boundedness_proof}, \ref{fig:lyap_level_sets}, and \ref{fig:lyap_made_bounded} are illustrations of the proof of Lemma \ref{lem:bounded_solutions} on the boundedness of the solutions of the ODE \eqref{ode:exp}. Figure \ref{fig:glob_center_proof} supports the proof of Theorem \ref{thm:global_center} on the characterization of global centers. 

\begin{figure}
\begin{center}
\begin{tabular}{ccc}
\includegraphics[scale=.30]{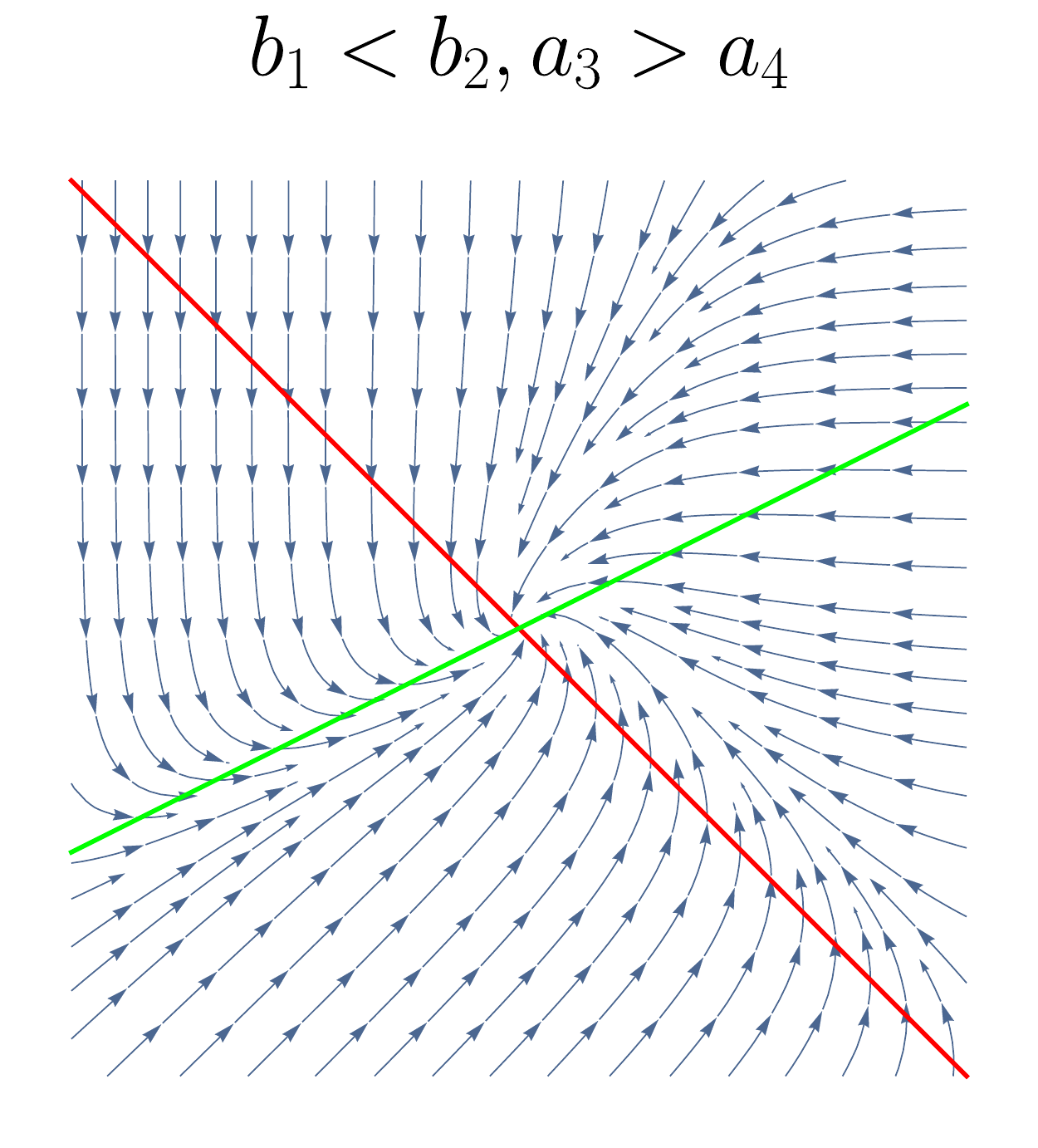} &
\includegraphics[scale=.30]{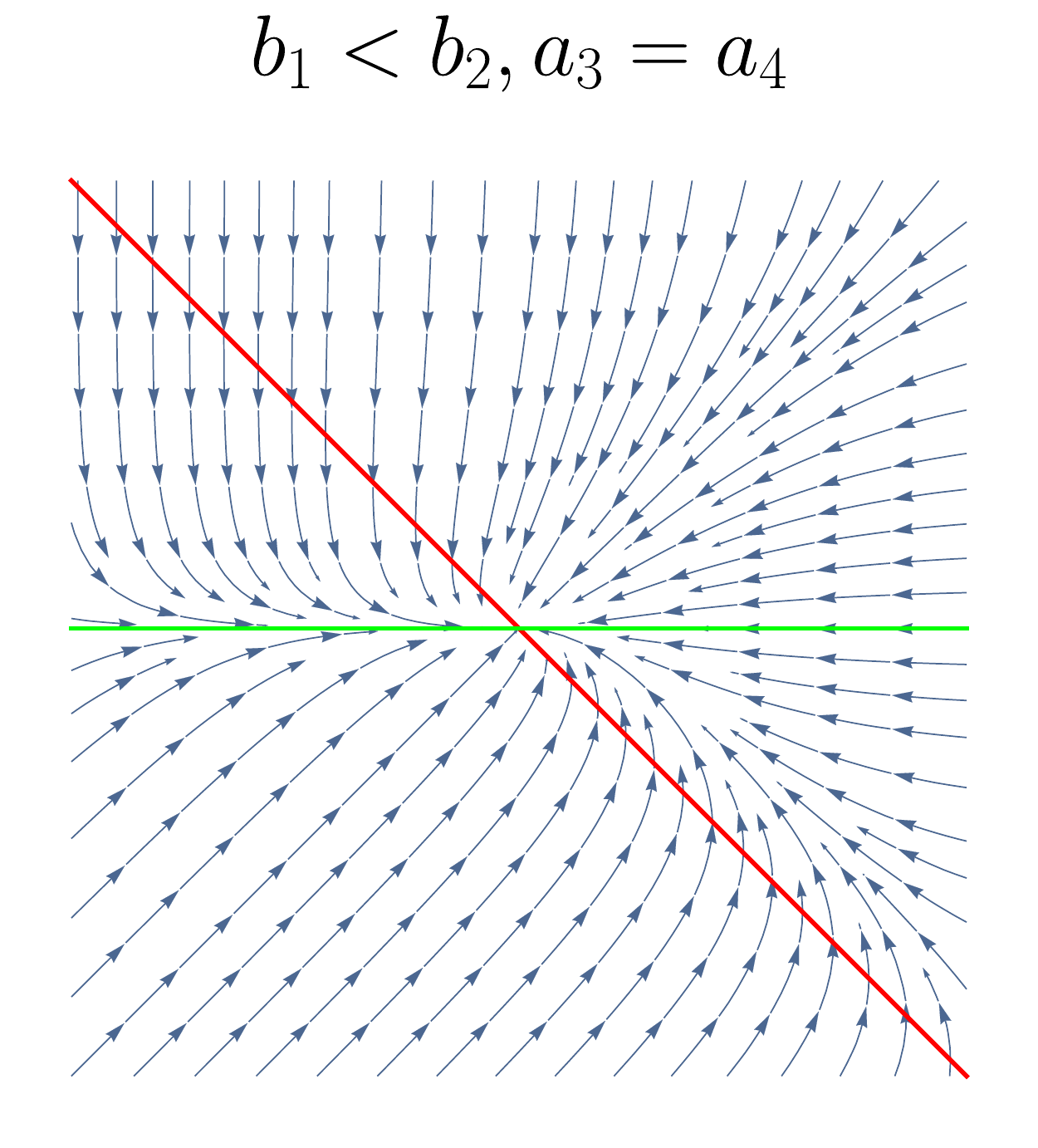} &
\includegraphics[scale=.30]{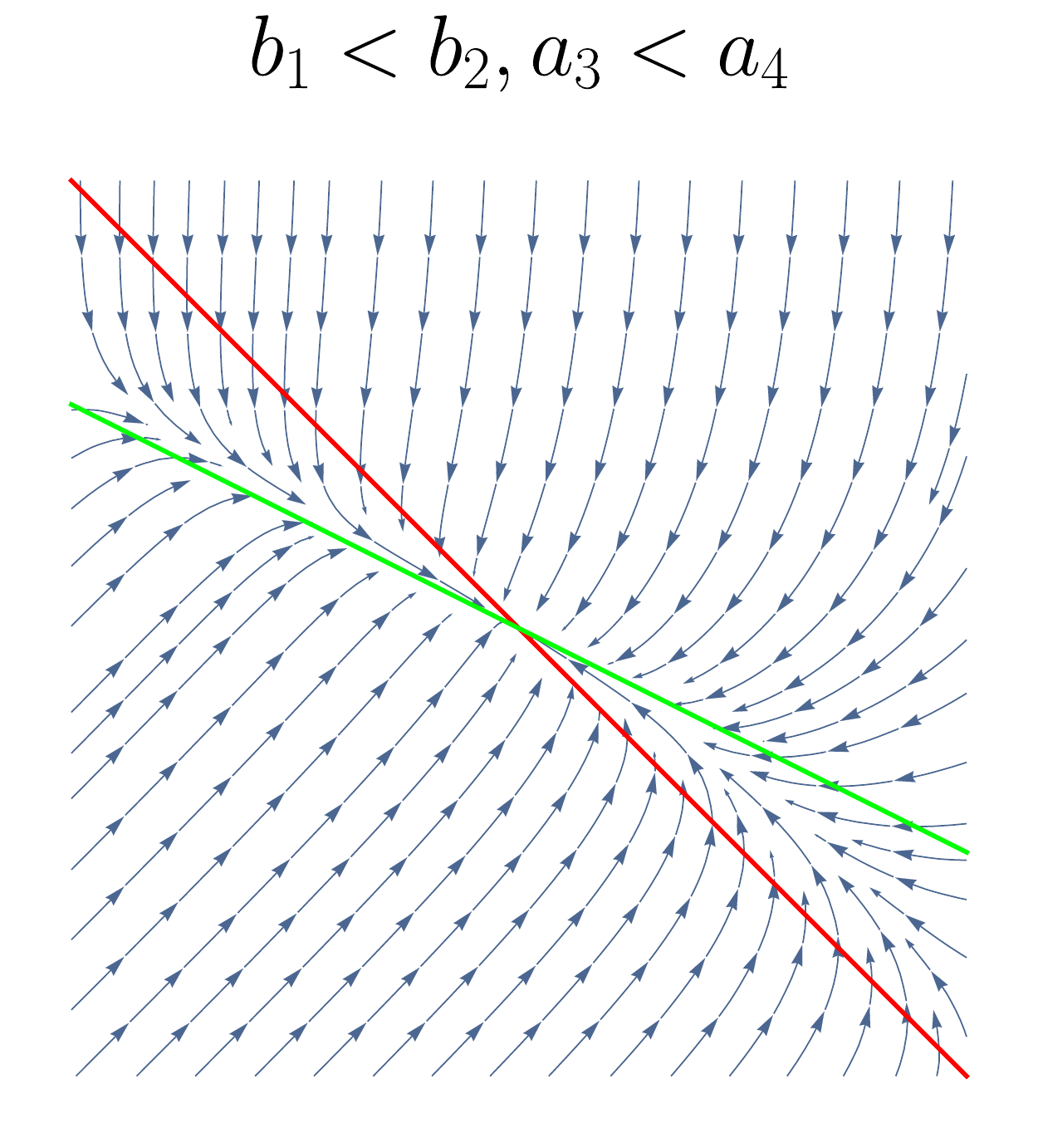} \\
\includegraphics[scale=.30]{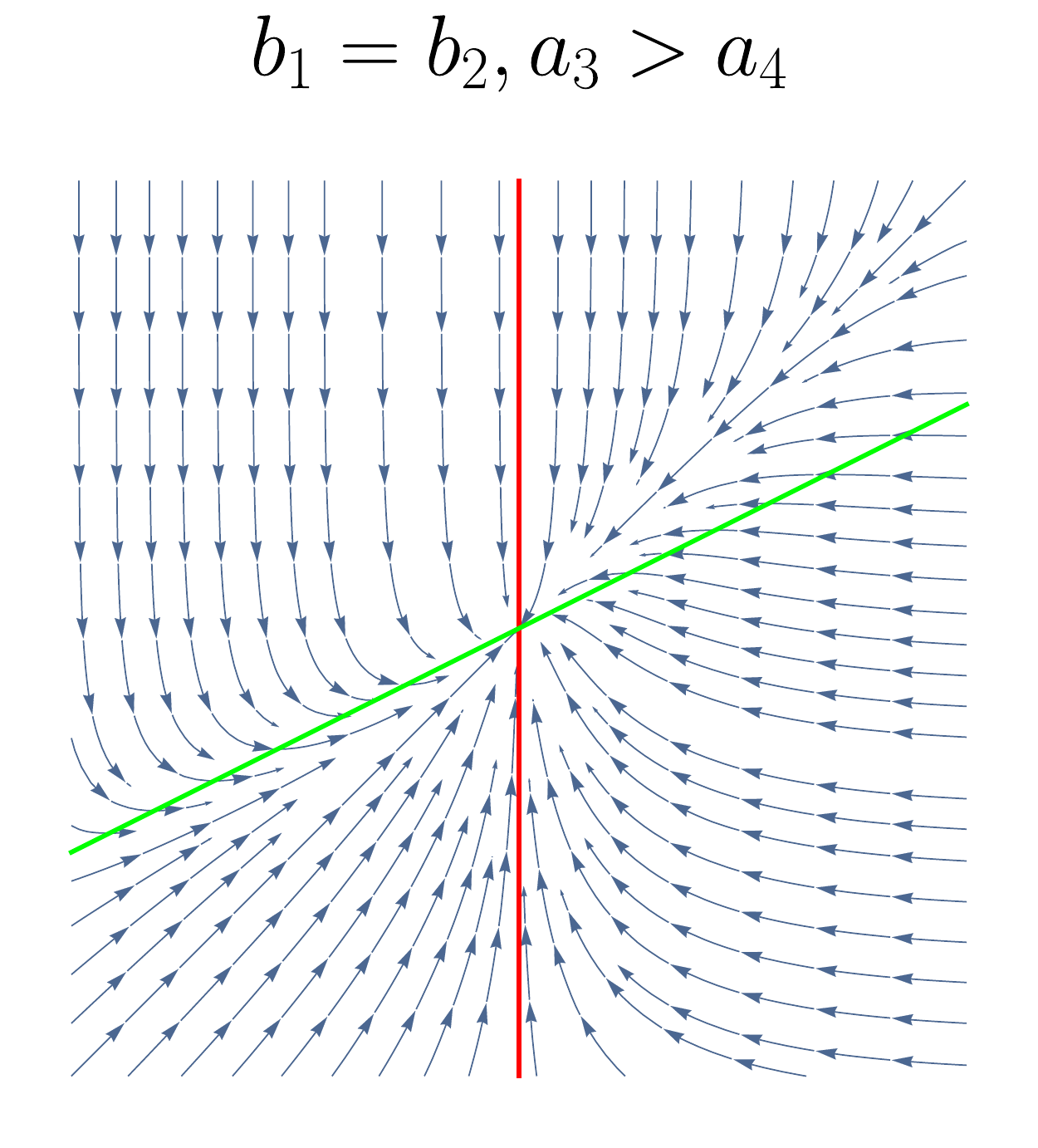} &
\includegraphics[scale=.30]{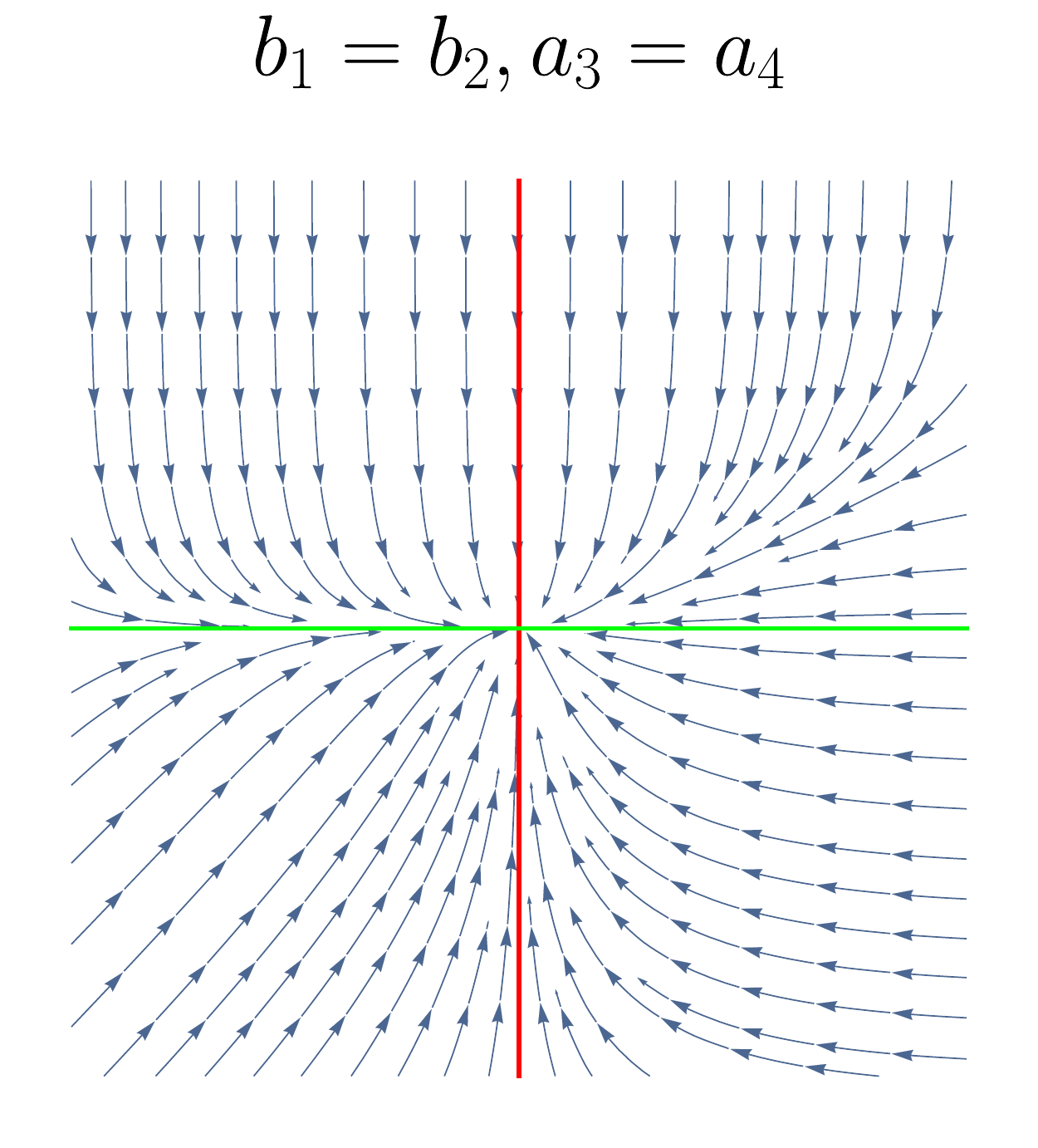} &
\includegraphics[scale=.30]{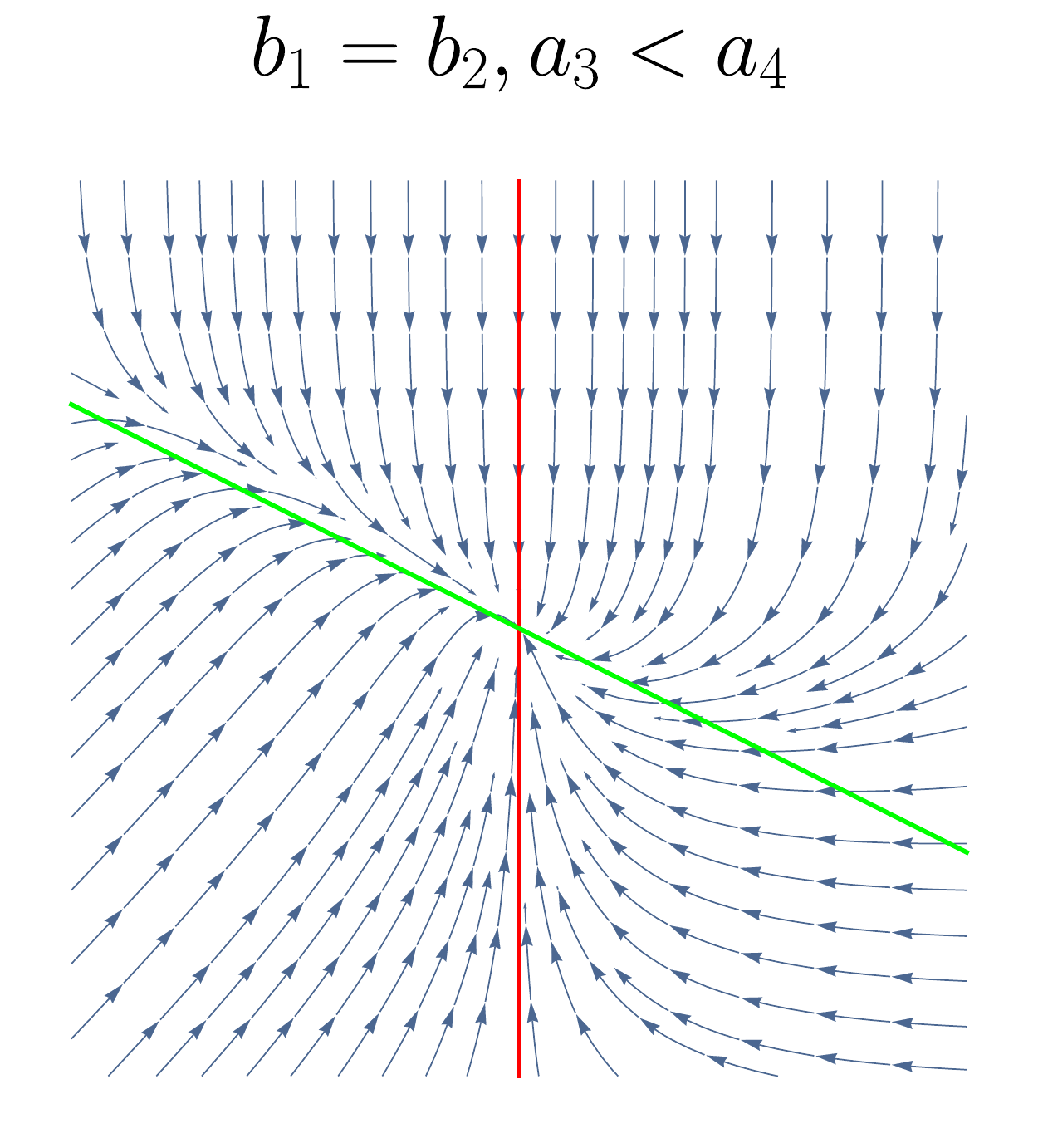} \\
\includegraphics[scale=.30]{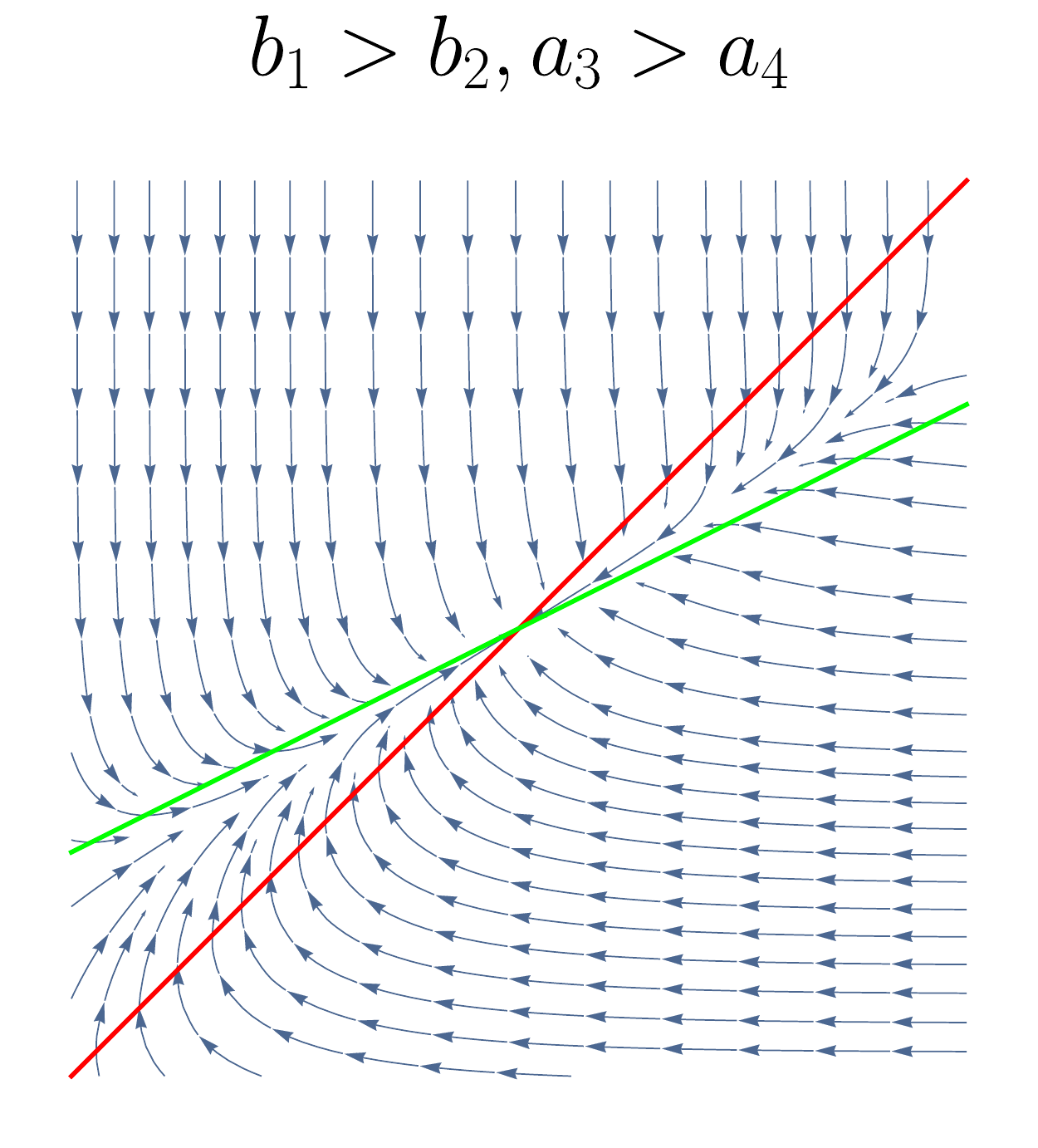} &
\includegraphics[scale=.30]{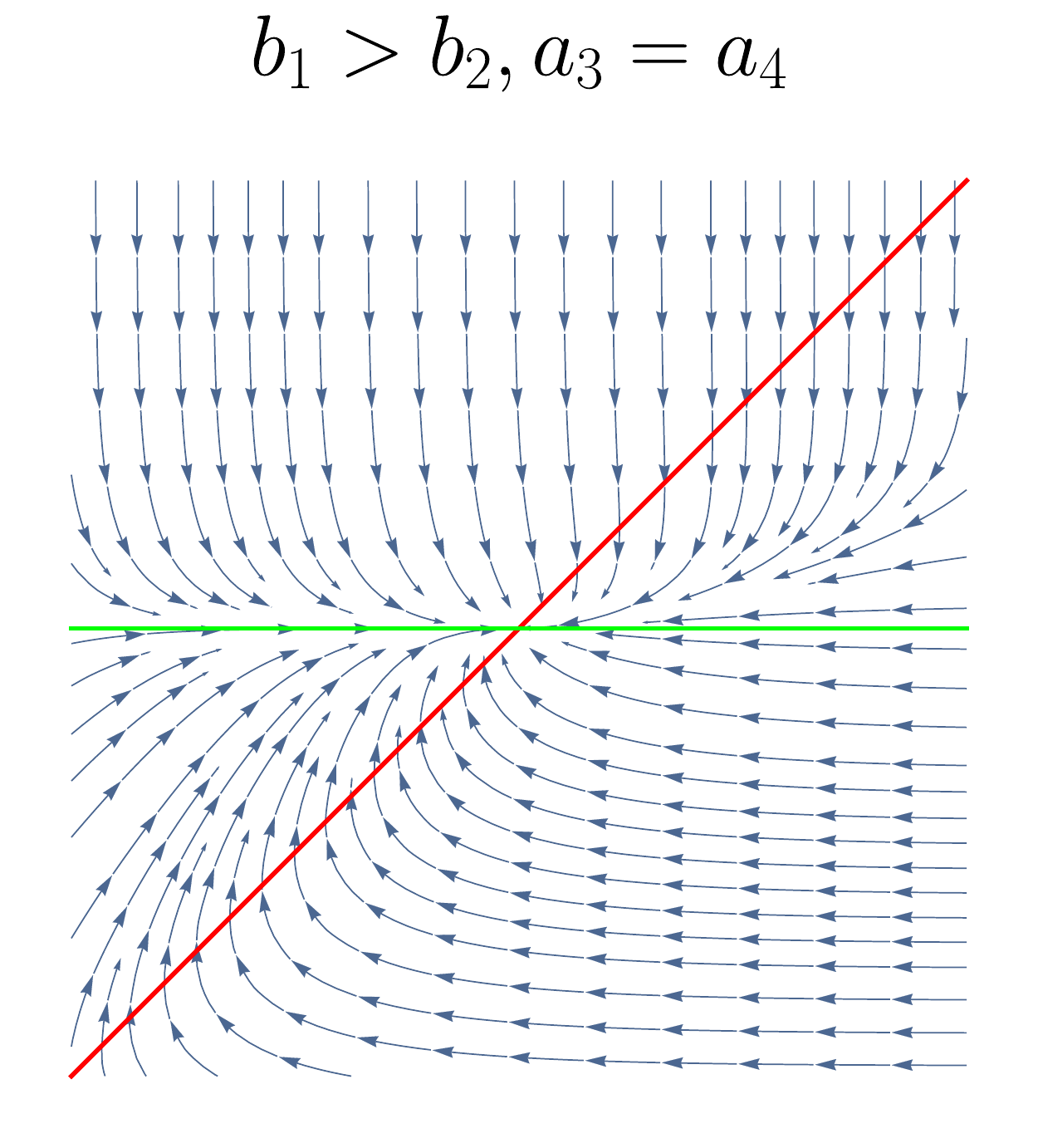} &
\includegraphics[scale=.30]{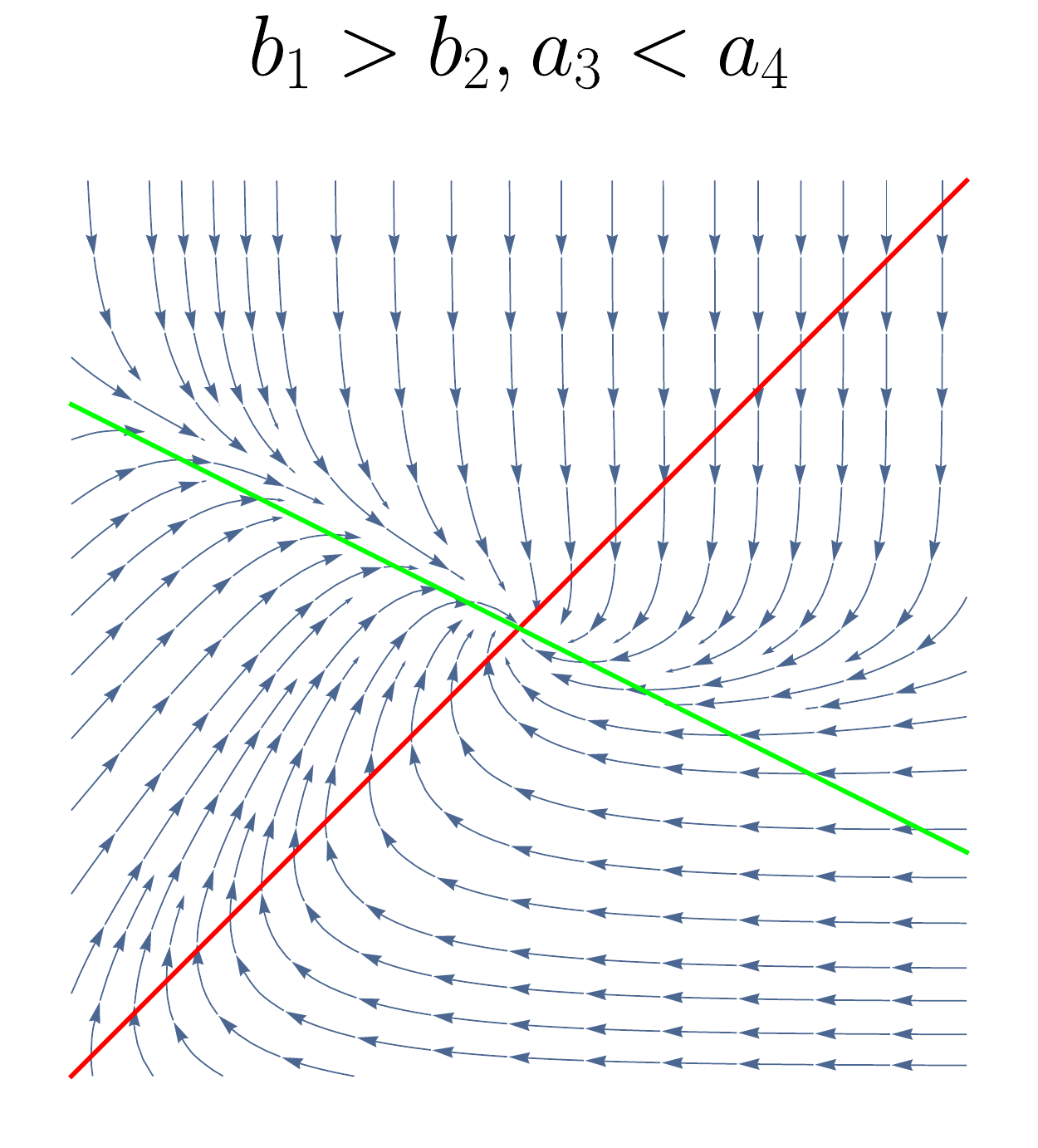} \\
\end{tabular}
\end{center}
\caption{Phase portraits of the ODE \eqref{ode:exp} in case $\det J > 0$ and both of the diagonal entries of $J$ are negative. As claimed in Lemma~\ref{lem:divergence}~(a), all solutions are bounded in positive time.
Seven cases are ultimately monotonic, the remaining two (top left and bottom right) can spiral, but only inwards.}
\label{fig:streamplots_very_stable_region}
\end{figure}

\clearpage

\begin{figure}[h]
\begin{center}
\begin{tabular}{c}
\includegraphics[scale=.49]{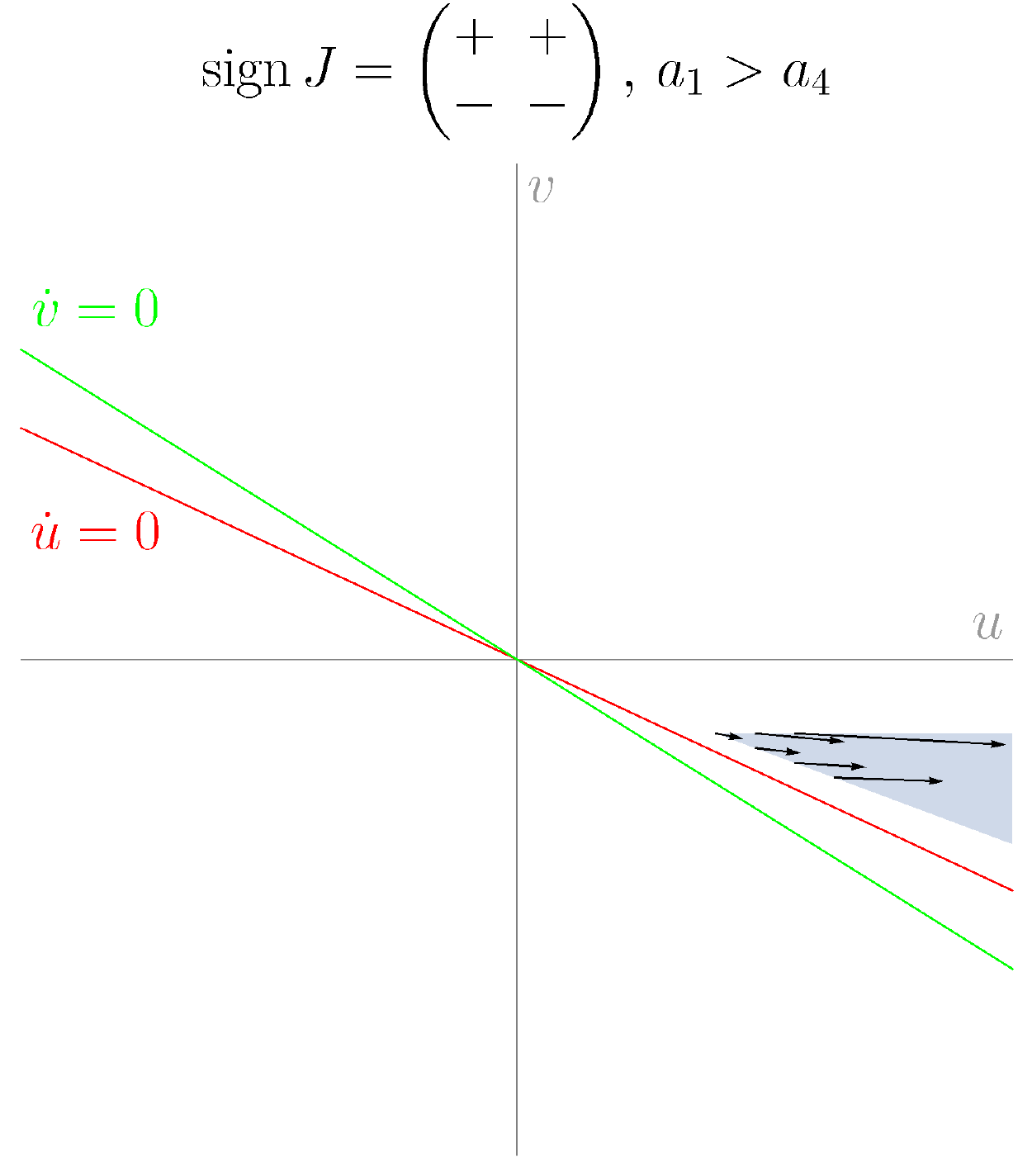} \\
\includegraphics[scale=.49]{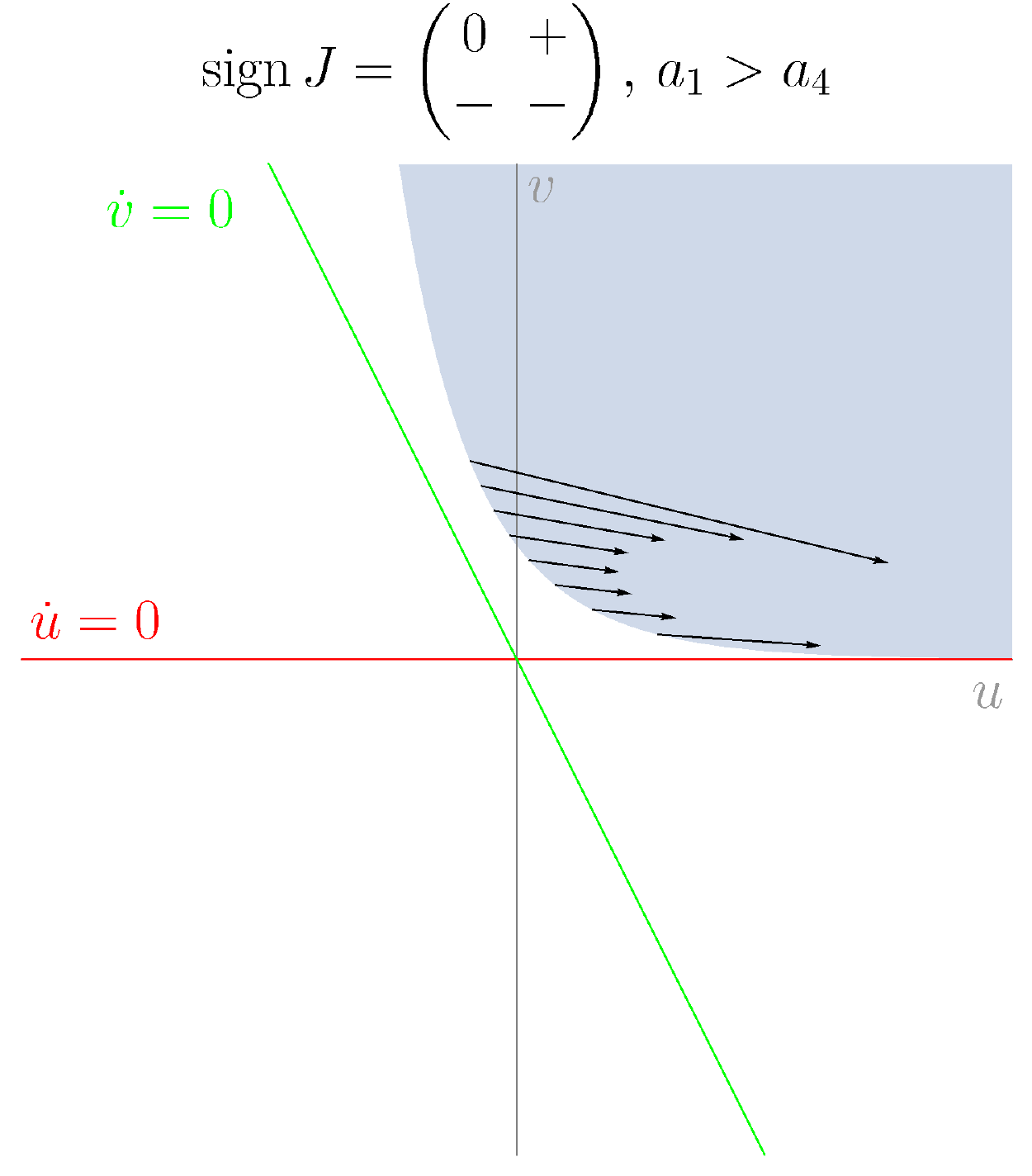} \\
\end{tabular}
\end{center}
\caption{The forward invariant sets used in the proofs of Lemma~\ref{lem:bounded_solutions}~(b1) and (b2), respectively,
to show the necessity of $a_3 \leq a_2 < a_1 \leq a_4$ (top panel) and $a_3 \leq a_2 = a_1 \leq a_4$ (bottom panel)
for the boundedness of the solutions of the ODE~\eqref{ode:exp}.}
\label{fig:boundedness_proof}
\end{figure}

\clearpage

\begin{figure}
\begin{center}
\begin{tabular}{ccc}
\includegraphics[scale=.30]{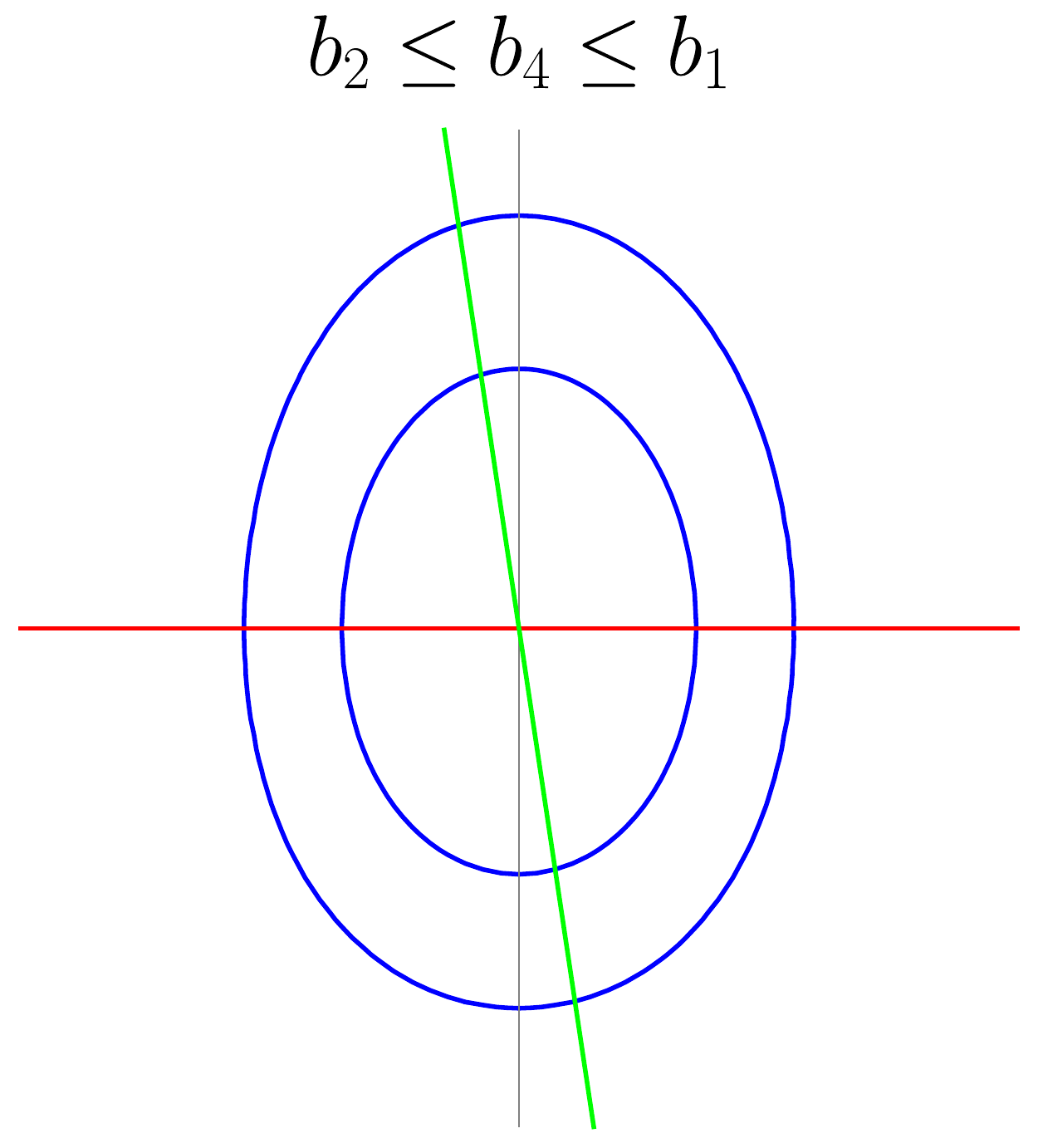} &
\includegraphics[scale=.30]{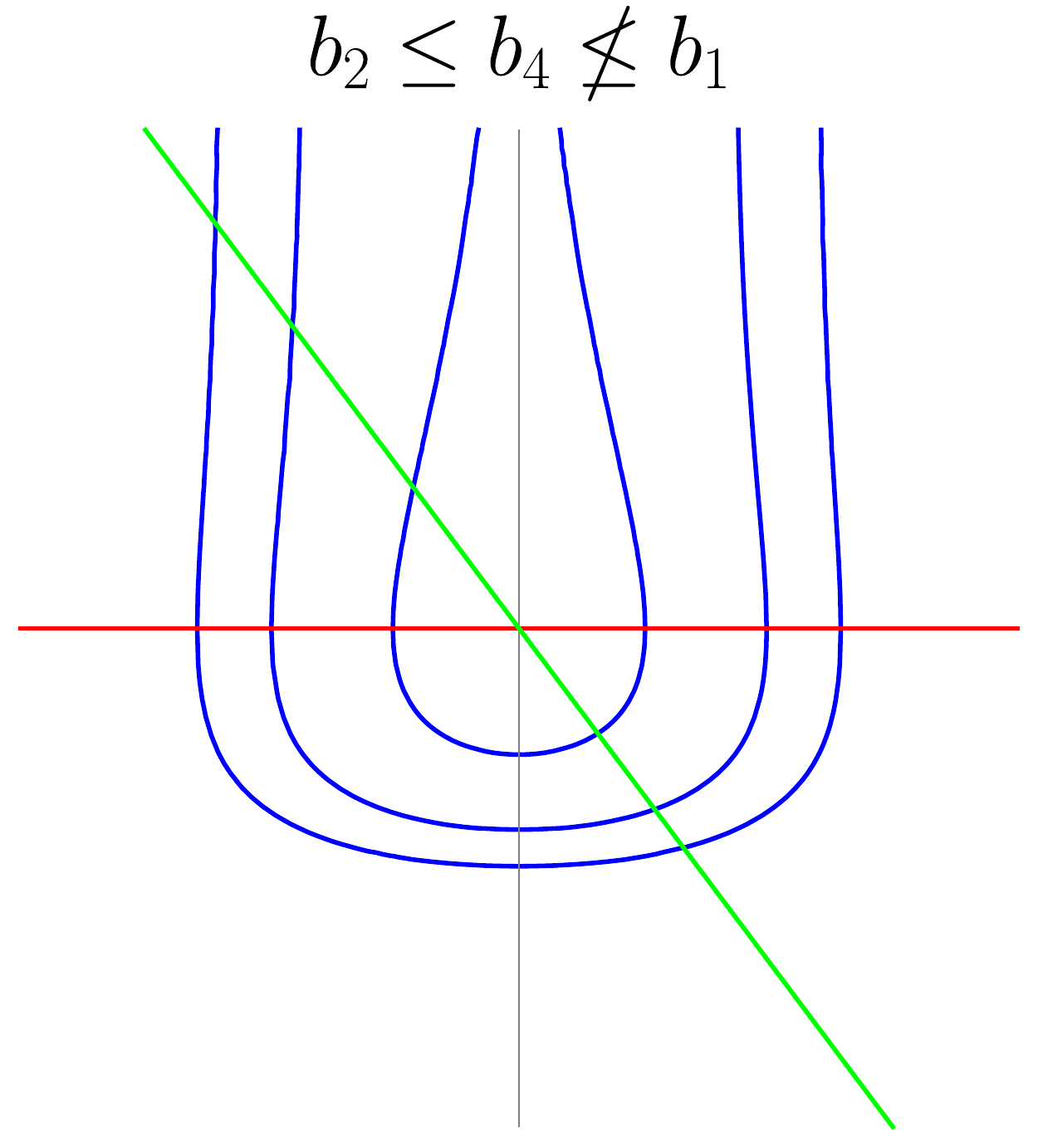} &
\includegraphics[scale=.30]{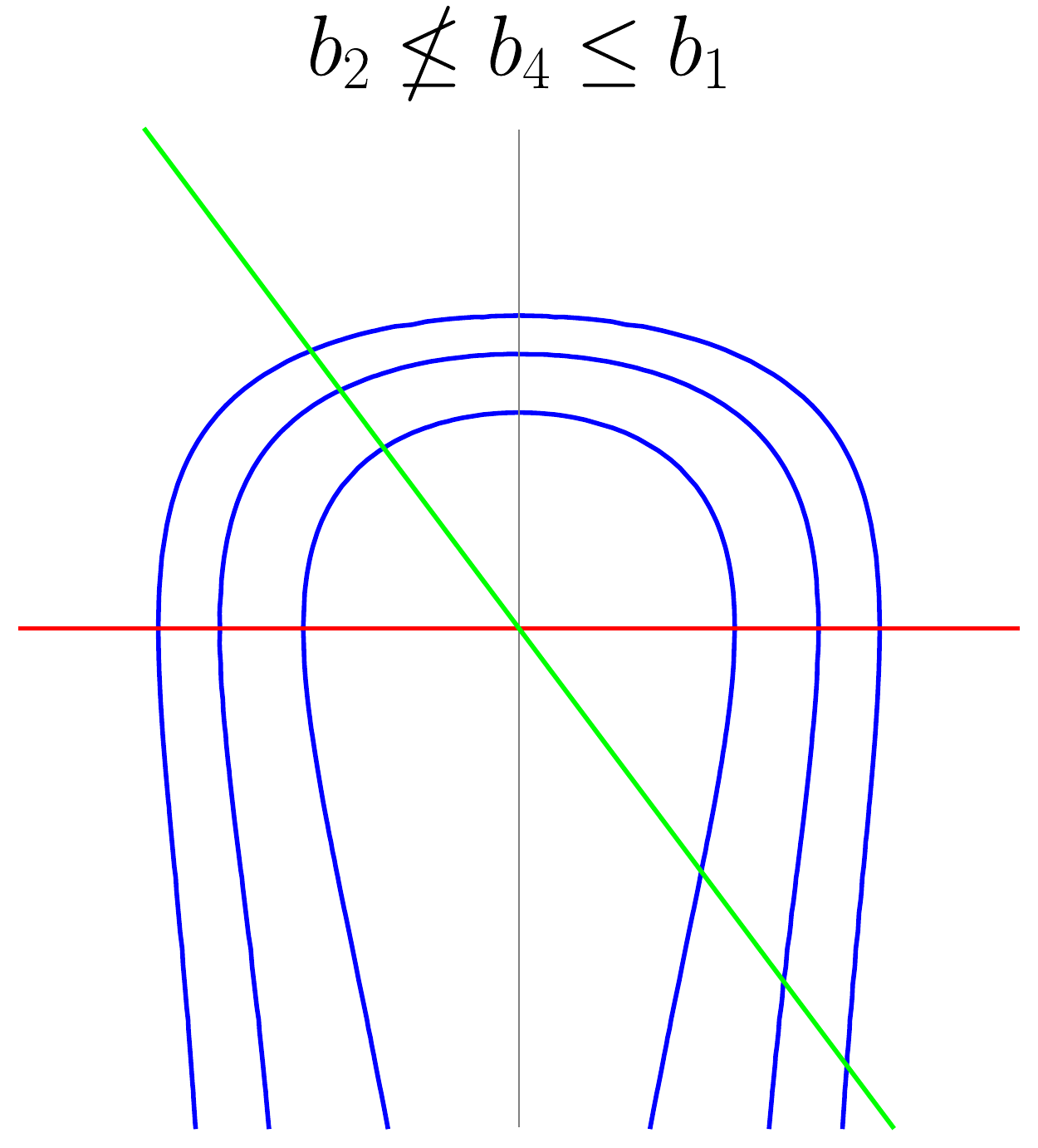} \\
\end{tabular}
\end{center}
\caption{The level sets of the Lyapunov function used to show the sufficiency of $a_3 \leq a_2 = a_1 \leq a_4$ for the boundedness of the solutions of the ODE~\eqref{ode:exp} in Lemma~\ref{lem:bounded_solutions} (b2).}
\label{fig:lyap_level_sets}
\end{figure}

\begin{figure}
\begin{center}
\begin{tabular}{ccc}
\includegraphics[scale=.30]{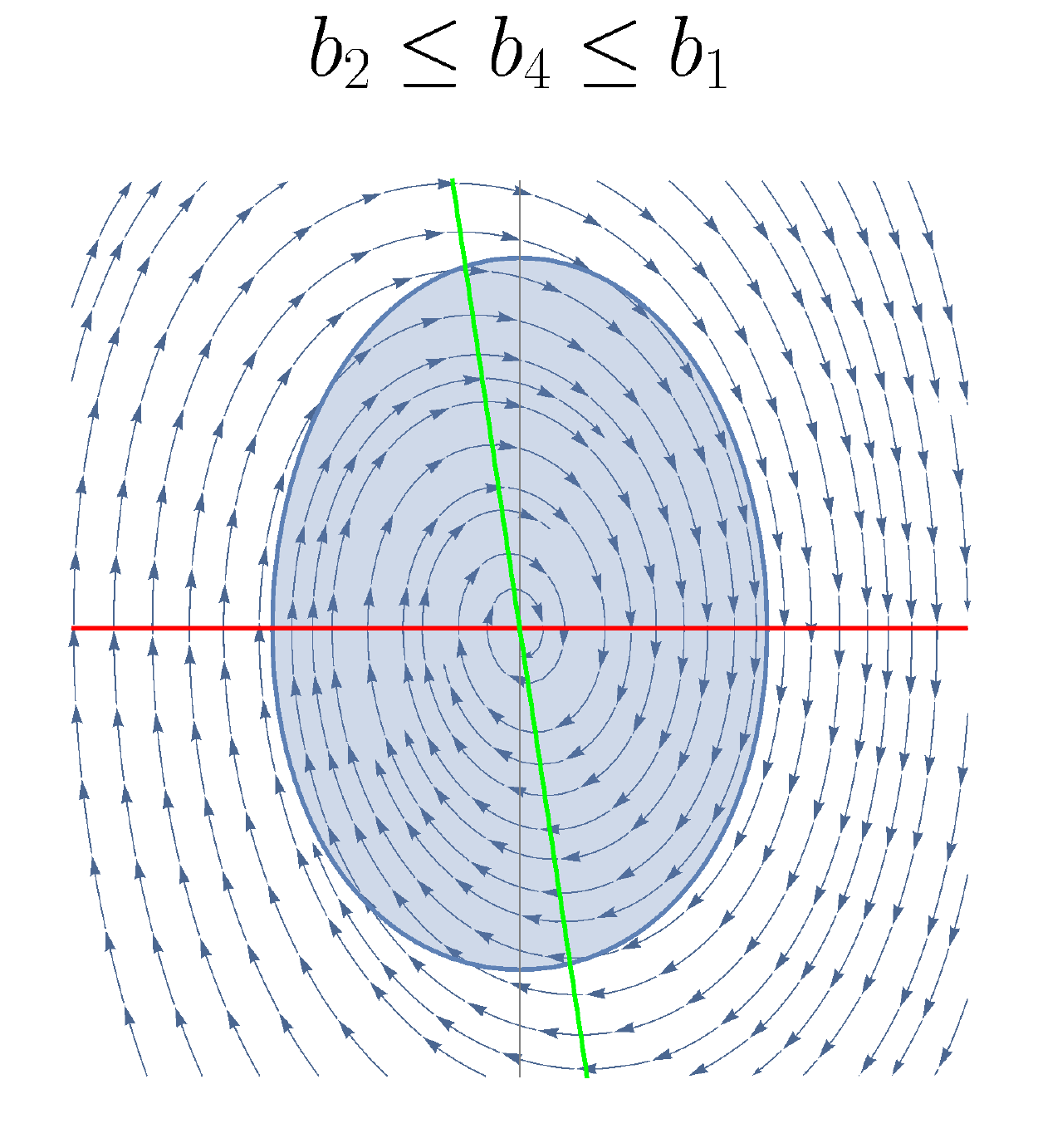} &
\includegraphics[scale=.30]{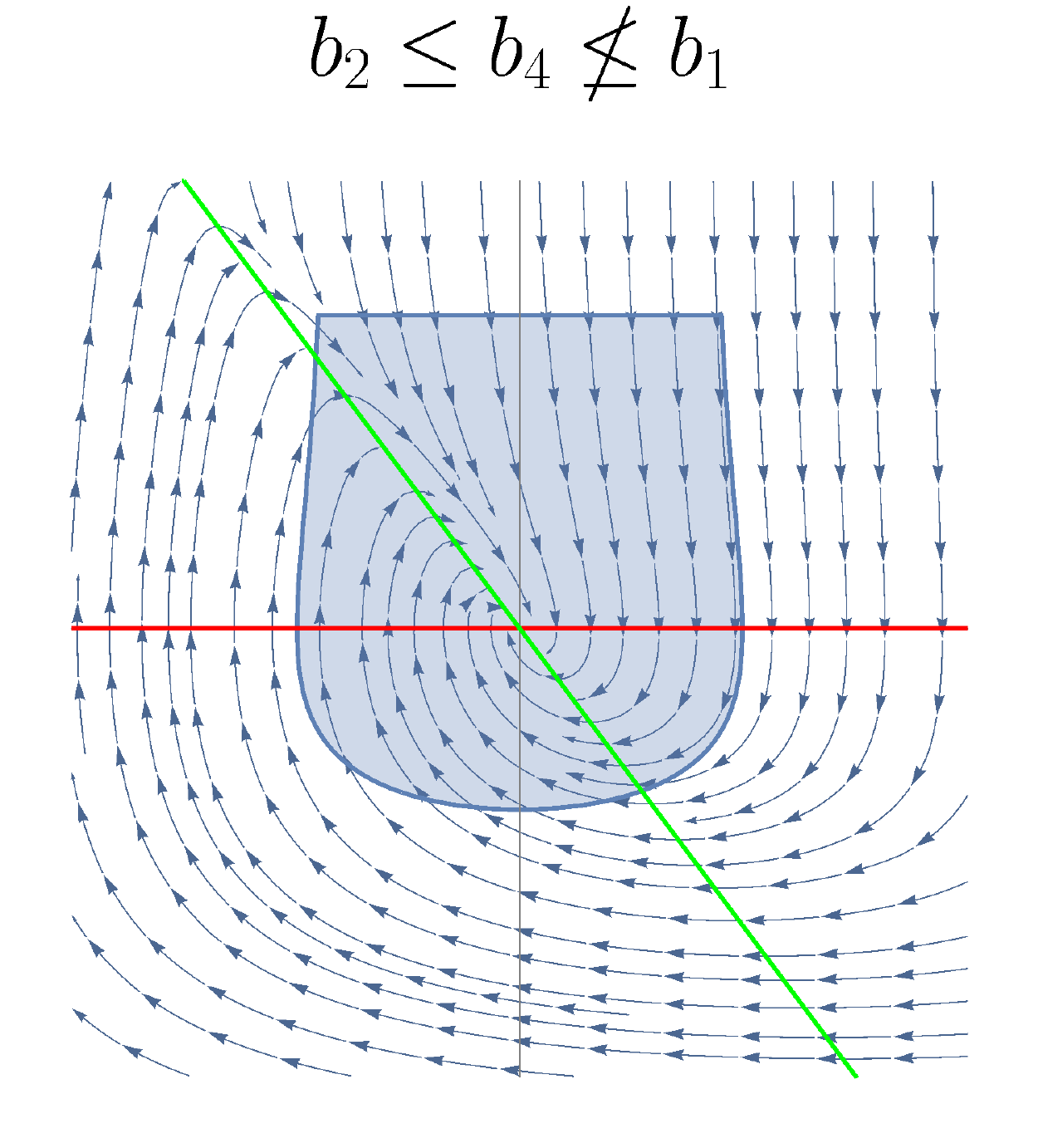} &
\includegraphics[scale=.30]{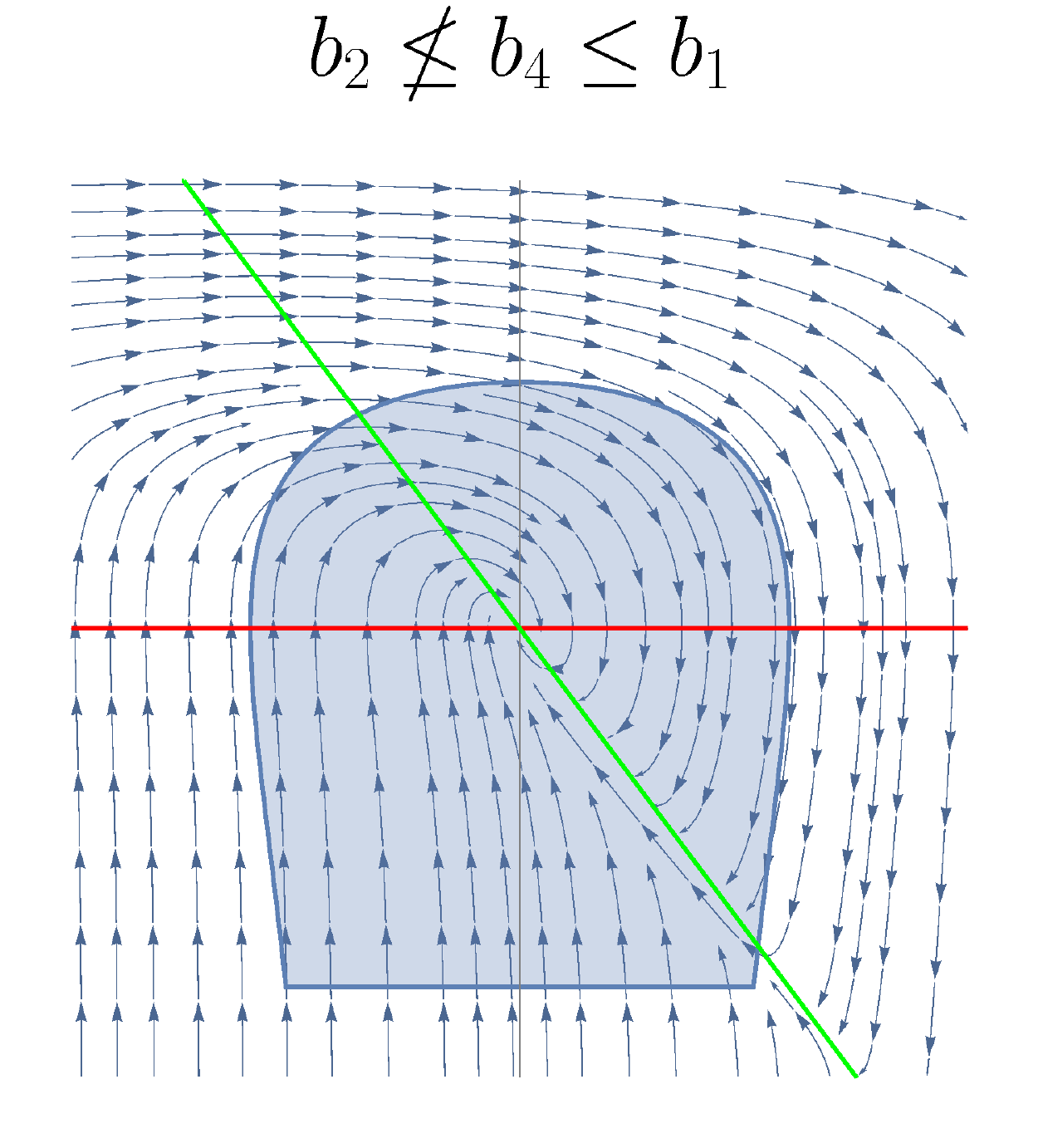} \\
\end{tabular}
\end{center}
\caption{The bounded forward invariant sets used to show the sufficiency of $a_3 \leq a_2 = a_1 \leq a_4$ for the boundedness of the solutions of the ODE~\eqref{ode:exp} in Lemma~\ref{lem:bounded_solutions} (b2).}
\label{fig:lyap_made_bounded}
\end{figure}

\clearpage

\begin{figure}
\begin{center}
\begin{tabular}{c}
\includegraphics[scale=.6]{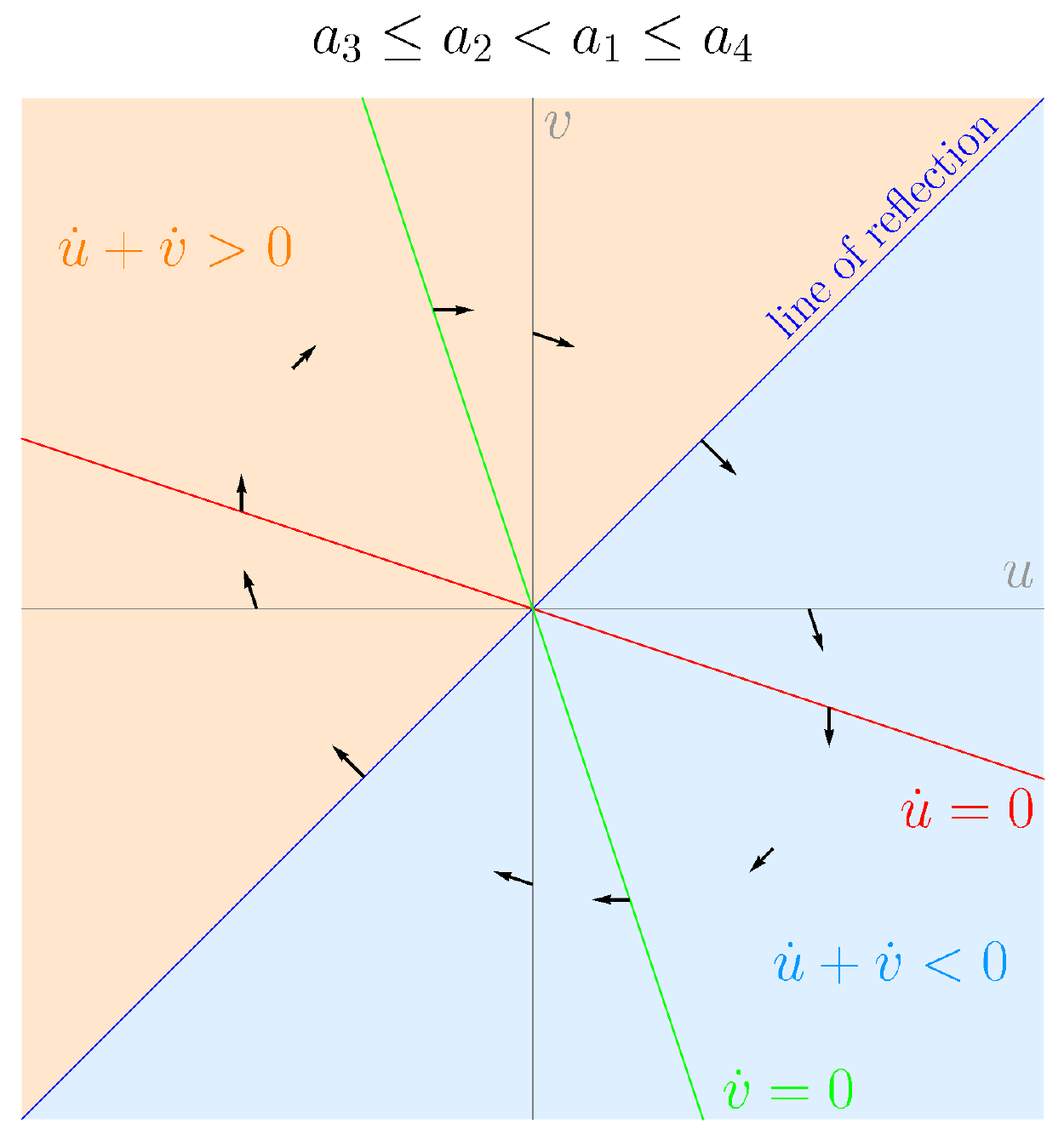} \\
\includegraphics[scale=.6]{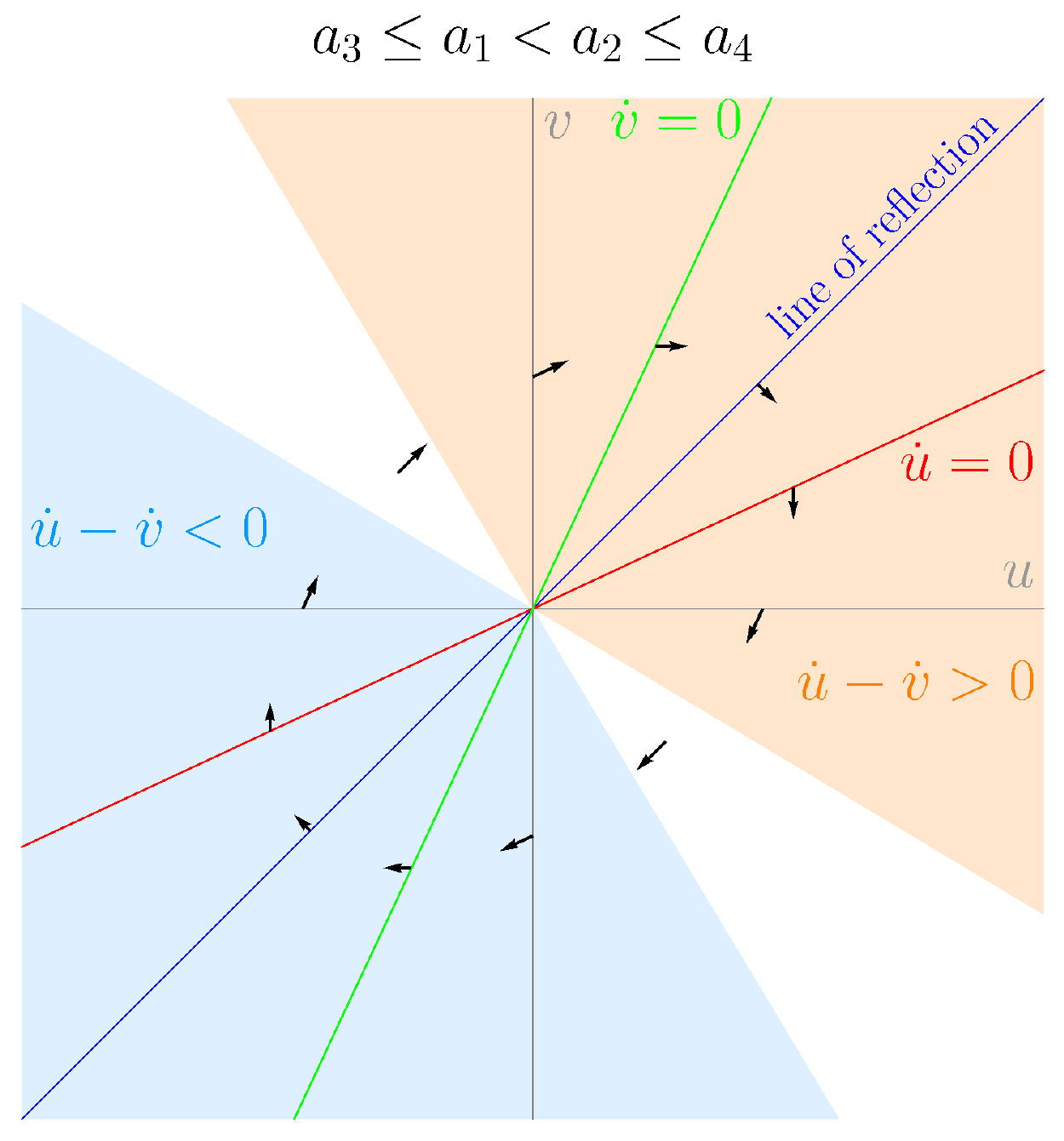} \\
\end{tabular}
\end{center}
\caption{Illustration of the proof of Theorem \ref{thm:global_center}, case R1,
to show the sufficiency of $a_3 \leq a_2 < a_1 \leq a_4$ (and $a_3 \leq a_1 < a_2 \leq a_4$, respectively) for the origin being a global center of the ODE~\eqref{ode:exp}. 
Both panels display the nullcline geometry, the sign structure of the vector field, the line of reflection, and the signs of $\dot u + \dot v$ and $\dot u - \dot v$.}
\label{fig:glob_center_proof}
\end{figure}

%

\clearpage

\bibliographystyle{AIMS}
\bibliography{stability}

\end{document}